 \documentclass[10pt,reqno]{amsart}
 \usepackage[a4paper,twoside]{geometry}

\usepackage{amsthm}
 \usepackage{amsmath}
 \usepackage{amsfonts}
 \usepackage{amssymb}
 \usepackage{amscd}
 \usepackage{url}
 \usepackage[T1]{fontenc}
 \usepackage[utf8x]{inputenc}
 \usepackage[english]{babel}
 \usepackage[mathscr]{eucal}
 \usepackage{lmodern}
 \usepackage{bera}
 \usepackage{float} 
 \usepackage{bold-extra}
 \usepackage{textcomp}
 \usepackage{graphicx}
 \usepackage{caption}
\usepackage{subcaption} 
 \usepackage[dvipsnames,table]{xcolor}
 \usepackage{mathtools} 
 \usepackage{enumitem}
 \usepackage{tikz} 
 \usepackage{tikz-cd}
 \usepackage{wasysym}
 \usepackage{todonotes}
 \usepackage{fancybox}
 \usepackage{fancyhdr}
 \usepackage[nodisplayskipstretch]{setspace}
 \usepackage[linkcolor=BrickRed,citecolor=darkblue,colorlinks=true]{hyperref}

 \theoremstyle{definition}  
  \newtheorem{defn}{Definition}[section]
  \newtheorem{eg}[defn]{Example}
  
  \newtheorem{obsr}[defn]{Observation}
  
  \theoremstyle{plain}  
  \newtheorem{thm}[defn]{Theorem}
  \newtheorem{lem}[defn]{Lemma}
  \newtheorem{prop}[defn]{Proposition}
  \newtheorem{cor}[defn]{Corollary}

  \theoremstyle{remark}  
  \newtheorem{note}[defn]{Note}
  \newtheorem{rmk}[defn]{Remark}

 \renewcommand{\sf}[1]{\textsf{#1}}

 \newcommand{\mbb}[1]{\mathbb{#1}}
 \newcommand{\mcl}[1]{\mathcal{#1}}

 \newcommand{\msc}[1]{\mathscr{#1}}

 \newcommand{\ol}[1]{\overline{#1}}
 \newcommand{\ul}[1]{\underline{#1}}
 \newcommand{\wtilde}[1]{\widetilde{#1}}

 \newcommand{\abs}[1]{\left\lvert#1\right\rvert}

 \newcommand{\norm}[1]{\left\lVert#1\right\rVert}
 \newcommand{\bnorm}[1]{\bigl\lVert#1\bigr\rVert}

 \newcommand{\M}[1]{\mbb{M}_{#1}}

 \newcommand{\B}[1]{\msc{B}({#1})}

 \newcommand{\ranko}[2]{|{#1}\rangle\langle{#2}|}

 \newcommand{\ip}[1]{\langle#1\rangle}
 \newcommand{\bip}[1]{\bigl\langle#1\bigr\rangle}
 \newcommand{\Bip}[1]{\Bigl\langle#1\Bigr\rangle}

 \newcommand{\ran}[1]{\sf{range}(#1)}
 
 \renewcommand{\ker}[1]{\sf{ker}(#1)}

 \newcommand{\mscriptsize}[1]{{\setlength{\arraycolsep}{.3ex}\text{\scriptsize$#1$}}}

 \newcommand{\Matrix}[1]{\begin{bmatrix}#1\end{bmatrix}}
 
 \newcommand{\sMatrix}[1]{\mscriptsize{\Matrix{#1}}}

 \DeclareMathOperator{\T}{\sf{T}}
 \DeclareMathOperator{\tr}{\sf{tr}}
 
 \DeclareMathOperator{\id}{\sf{id}}

 \DeclareMathOperator{\lspan}{\sf{span}}

 \numberwithin{equation}{section}
 \allowdisplaybreaks[4] 
 \setlist[enumerate]{font=\upshape,noitemsep, topsep=0pt} 
 \setlist[itemize]{noitemsep, topsep=0pt}
  
 \setuptodonotes{size=\tiny, color={red!70!green!20}}
 \definecolor{darkblue}{rgb}{0.0,0.0,0.3}

\makeatletter
\@namedef{subjclassname@2020}{\textup{2020} Mathematics Subject Classification}
\makeatother

\begin{document}

\title[$k$-Entanglement breaking maps]{Mapping Cone of $k$-Entanglement Breaking Maps}

\author{Repana Devendra}
\address{Department of Mathematics, Indian Institute of Technology Madras, Chennai, Tamilnadu 600036, India}
\email{r.deva1992@gmail.com}

\author{Nirupama Mallick}
\address{Chennai Mathematical Institute, H1, SIPCOT IT Park, Siruseri, Kelambakkam 603103, India}
\email{niru.mallick@gmail.com}

\author{K. Sumesh}
\address{Department of Mathematics, Indian Institute of Technology Madras, Chennai, Tamilnadu 600036, India}
\email{sumeshkpl@gmail.com, sumeshkpl@iitm.ac.in}

\date{\today}

\begin{abstract}
 In \cite{CMW19}, the authors introduced $k$-entanglement breaking linear maps to understand the entanglement breaking property of completely positive maps on taking composition. In this article, we do a systematic study of  $k$-entanglement breaking maps. We prove many equivalent conditions for a $k$-positive linear map to be $k$-entanglement breaking, thereby study the mapping cone structure of $k$-entanglement breaking maps.  We discuss examples of $k$-entanglement breaking maps and some of their significance. As an application of our study, we characterize the completely positive maps that reduce Schmidt number on taking composition with another completely positive map. 
\end{abstract}



\maketitle

\tableofcontents

\section{Introduction}

 Completely positive maps that are PPT or entanglement breaking are of great importance in Quantum Information theory. The problem of deciding whether a given linear map is entanglement breaking or not is computationally a hard one. Entanglement breaking maps are known to be PPT-maps, but the converse is not true in general. It is a well-known conjecture (\cite{RJKHW12}) that the square of a PPT channel is entanglement breaking. To understand the entanglement property of maps on taking composition, in \cite{CMW19}, the notion of $k$-entanglement breaking maps is introduced, where $k\in\mbb{N}$.  In this article, we study $k$-entanglement breaking maps in more detail. We generalize several known results about entanglement breaking maps into the setup of $k$-entanglement breaking maps.   We hope that our study will help understanding entanglement breaking maps better and hence be of substantial interest, especially in quantum physics.
 
  We organize this article into seven sections. In Section \ref{sec-Notn-Prel}, we recall some notations, basic definitions and known-results useful for later sections. In Section \ref{sec-nEB}, we establish various characterizations of $k$-entanglement breaking maps.   Making use of these,  we prove (Theorem \ref{thm-mapp-cone}) that  the set $\mcl{EB}_k$ and $\mcl{EBCP}_k$ of all $k$-entanglement breaking linear maps and  $k$-entanglement breaking completely positive maps, respectively, form mapping cones.  Further, we show that (Theorem \ref{thm-nEB-untypical}) both the cones serve as examples of non-symmetric,  untypical mapping cones. 

Section \ref{sec-eg} discuss examples of $k$-entanglement breaking maps and  some of their significance in determining separability (Theorem \ref{thm-sep-necc}) and entanglement (Corollary \ref{cor-EB-necc}). Though there is no close relation between PPT-maps and $k$-entanglement breaking linear maps (Proposition \ref{prop-PPT-non2EB} and Example \ref{eq-nEB-nonCP}), Theorem \ref{thm-nEB-suff-condn} provides a sufficient condition for a trace preserving positive map (in particular, PPT-map) to becomes a $k$-entanglement breaking map.  

 As an application, in Section \ref{sec-SN-red-CP}, we exhibit the Schmidt number reducing property of $2$-entanglement breaking completely positive maps. Motivated from work done in \cite{RJP18}, we ask the following question, which is closely related to PPT-square conjecture: Suppose $\Phi$ is a non-entanglement breaking completely positive map on $\M{d}$. Under what sufficient conditions on $\Phi$ does there exist $N\in\mbb{N}$ such that $\Phi^N$ is entanglement breaking? In such cases, $1=SN(\Phi^N)<SN(\Phi)$. This inequality motivates us to search for those completely positive maps, which reduce the Schmidt number after composing a finite number of times.  In Theorem \ref{thm-2EB-char}, we show that precisely those completely positive maps that are $2$-entanglement breaking reduce the Schmidt number on composing with another completely positive map. Further, if $\Phi$ is a $k$-entanglement breaking completely positive map, Corollary \ref{cor-nEB+2EB-2} provides an upper bound for $N$.  
 
  In Section \ref{sec-Maj}, we discuss a majorization result for $k$-entanglement breaking maps. Section \ref{sec-disc} discuss one open problem. In Appendix we prove the separability of a particular class of positive matrices, which we use in Section  \ref{sec-eg}. As the proof involves few technical lemmas,  we write it as a separate section.

\section{Notation and preliminaries}\label{sec-Notn-Prel}
 Throughout this article, we fix $d, d_1, d_2, d_3\in\mbb{N}$. Unless mentioned otherwise, $\{e_i\}_{i=1}^d\subseteq\mbb{C}^d$ always denotes the standard orthonormal basis. We let $\M{d_1\times d_2}$ denote the space of all $d_1\times d_2$ complex matrices and $I=I_d\in\M{d}=\M{d\times d}$ be the diagonal matrix with diagonals equals $1$.  By writing $A=[a_{ij}]\in\M{d_1\times d_2}$, we mean $A$ is a $d_1\times d_2$ complex matrix with $(i, j)^{th}$ entry equals $a_{ij}$ for all $1\leq i\leq d_1, 1\leq j\leq d_2$.   Further, we let  $\T=\T_d: \M{d}\to \M{d}$ denote the transpose map, $\tr: \M{d}\to \mbb{C}$ the trace map, and $\id=\id_d:\M{d}\to\M{d}$ the identity map. The cone of all positive (semidefinite) matrices in $\M{d}$ is denoted by $\M{d}^+$. If $A\in\M{d}^+$, then we write $A\geq 0$. Given $x\in\mbb{C}^{d_1}, y\in\mbb{C}^{d_2}$, define the mapping $\ranko{x}{y}:\mbb{C}^{d_2}\to\mbb{C}^{d_1}$ by $\ranko{x}{y}(z):=x\ip{y,z}$ for all $z\in\mbb{C}^{d_2}$. Note that $\ranko{x}{x}\geq 0$.
 
 We let $\Omega_d=\sum_{i=1}^d e_i\otimes e_i\in\mbb{C}^d\otimes\mbb{C}^d$. Given a unit vector  $\xi\in\mbb{C}^{d_1}\otimes\mbb{C}^{d_2}$ there always exist orthonormal sets $\{u_i\}_{i=1}^d\subseteq\mbb{C}^{d_1}$ and $\{v_i\}_{i=1}^d\subseteq\mbb{C}^{d_2}$, and scalars $\{\lambda_i\}_{i=1}^d\subseteq[0,1]$ with $\sum_i\lambda_i^2=1$ such that 
  \begin{align}\label{equ-schmdt-rnk}
     \xi=\sum_{i=1}^d\lambda_i (u_i\otimes v_i),
  \end{align}        
 where $d=\min\{d_1,d_2\}$. A decomposition of the form \eqref{equ-schmdt-rnk} is called a  \emph{Schmidt decomposition} (\cite{Sch06}) of the vector $\xi$.  The number of non-zero coefficients in any two Schmidt decomposition of $\xi$ is the same, and this unique number is called the \emph{Schmidt rank} of $\xi$ and is denoted by $SR(\xi)$. 
 
\begin{defn} 
  Let $X\in(\M{d_1}\otimes\M{d_2})^+$.  If there exist $A_i\in\M{d_1}^+$ and $B_i\in\M{d_2}^+, 1\leq i\leq n,$ such that $X=\sum_{i=1}^n A_i\otimes B_i$, 
  then $X$ is said to be \emph{separable}, otherwise called \emph{entangled}. 
\end{defn}
 Though sufficient and necessary conditions for separability are known (\cite{HHH96}), determining whether a given positive matrix is separable or not is very hard. Often one uses the Schmidt number techniques to handle this situation. The \emph{Schmidt number} (\cite{TeHo00}) of a positive matrix $X\in(\M{d_1}\otimes\M{d_2})^+$ is defined as  
           $$SN(X):=\min\Big\{\max\Big\{SR(\xi_i): X=\sum_{i=1}^ n\ranko{\xi_i}{\xi_i}\Big\}: \xi_i\in\mbb{C}^{d_1}\otimes\mbb{C}^{d_2}, n\geq 1\Big\}.$$
  Clearly $SN(\ranko{\xi}{\xi})=SR(\xi)$ for all $\xi\in\mbb{C}^{d_1}\otimes\mbb{C}^{d_2}$, and $SN(\sum_{i=1}^n X_i)\leq \max_i\{SN(X_i)\}$, where $X_i\in(\M{d_1}\otimes\M{d_2})^+,1\leq i\leq n$. Note that $X$ is separable if and only if $SN(X)=1$. From the definition of Schmidt number, $SN(X)\leq\min\{d_1,d_2\}$. 
   
A linear  map $\Phi:\M{d_1}\to\M{d_2}$ is said to be \emph{positive} if $\Phi(\M{d_1}^+)\subseteq \M{d_2}^+$. Given $k\geq 1$, $\Phi$ is said to be \emph{$k$-positive} if   $\id_k\otimes\Phi:\M{k}\otimes\M{d_1}\to\M{k}\otimes\M{d_2}$ is a positive map.  If both $\Phi$ and $\T\circ\Phi$ (equivalently $\Phi\circ\T$) are $k$-positive, then $\Phi$ is called a \emph{$k$-PPT map}.
 If $\Phi$ is $k$-positive for every $k\geq 1$,  then $\Phi$ is called a \emph{completely positive} (CP-)map. Every CP-map $\Phi:\M{d_1}\to\M{d_2}$ can be written as 
 \begin{align}\label{eq-Kraus-decomp}
     \Phi=\sum_{i=1}^n Ad_{V_i}
 \end{align}
 for some $V_1, V_2, \cdots, V_n\in\M{d_1\times d_2}$, where $Ad_{V}(X):=V^*XV$ for all $X\in\M{d_1}$ and $V\in\M{d_1\times d_2}$. A decomposition of $\Phi$ given in \eqref{eq-Kraus-decomp}  is called a \emph{Choi-Kraus decomposition} (\cite{Kra71,Cho75}) of $\Phi$, and $V_i$'s are called Choi-Kraus operators.  A linear map $\Phi:\M{d_1}\to\M{d_2}$ is said to be \emph{completely co-positive} (co-CP) if $\T\circ\Phi$ (equivalently $\Phi\circ\T$) is a CP-map. A linear map $\Phi$ that is both CP and co-CP  is called a \emph{PPT-map}. 
 
  Given $A=[a_{ij}]\in\M{d_1}$ and $B=[b_{ij}]\in\M{d_2}$, we let $A\otimes B=[a_{ij}B]$ and thereby identify $\M{d_1}\otimes\M{d_2}=\M{d_1}(\M{d_2})$. To every linear map $\Phi:\M{d_1}\to\M{d_2}$ associate the matrix  
 \begin{align*}
      C_\Phi:=(\id_{d_1}\otimes\Phi)(\sum_{i,j=1}^{d_1}E_{ij}\otimes E_{ij})
                  =[\Phi(E_{ij})]\in\M{d_1}\otimes\M{d_2}=\M{d_1}(\M{d_2}),
 \end{align*} 
 where $E_{ij}=\ranko{e_i}{e_j}\in\M{d_1},~1\leq i,j\leq d_1$, are the standard matrix units. The map $\Phi\mapsto C_\Phi$ is an isomorphism (called \emph{Choi-Jamiolkowski isomorphism} \cite{Cho75, Jam72}) from the space $\{\Phi:\M{d_1}\to\M{d_2}\mbox { linear map}\}$ onto the space $\M{d_1}\otimes\M{d_2}=\M{d_1}(\M{d_2})$. 
 Further, $\Phi$ is a CP-map if and only if $\Phi$ is $d_1$-positive if and only if $C_\Phi\in(\M{d_1}\otimes\M{d_2})^+$; and $\Phi$ is PPT if and only if $C_\Phi$ is a positive matrix with positive partial transpose (i.e., $(\id_{d_1}\otimes\T)(C_\Phi)\geq 0$).
  
 Suppose $\Phi:\M{d_1}\to\M{d_2}$ is a CP-map. Then the \emph{Schmidt number} of $\Phi$ is denoted and defined (\cite{ChKo06,Hua06}) as $SN(\Phi):=SN(C_\Phi)$. It is known (\cite{ChKo06}) that
  \begin{align*}
      SN(\Phi)=\min\Bigg\{\max\Big\{rank(V_i): \Phi=\sum_{i=1}^n Ad_{V_i}\Big\}: V_i\in\M{d_1\times d_2}, n\geq 1\Bigg\}.
  \end{align*}
 Given a linear map $\Phi:\M{d_1}\to\M{d_2}$, we let $\Phi^*$ denotes the dual of $\Phi$ w.r.t the Hilbert Schmidt inner product $\ip{X,Y}:=\tr(X^*Y)$. From Choi-Kraus decomposition, it follows that
    \begin{itemize}
     \item $SN(\Phi)=SN(\Phi^*)$,
     \item $SN(\Phi\circ\Psi)\leq\min\big\{SN(\Phi),SN(\Psi)\big\}$, 
     \item $SN(\Phi+\Psi)\leq\max\big\{SN(\Phi),SN(\Psi)\big\}$, 
   \end{itemize}
  where $\Phi,\Psi$ are CP-maps. 
  
\begin{defn}
 A linear map $\Phi:\M{d_1}\to\M{d_2}$ is said to be an \emph{entanglement breaking (EB)} if $(\id_k\otimes \Phi)(X)$ is separable for all $X\in (\M{k}\otimes\M{d_1})^+$ and $k\geq 1$. 
\end{defn}

 Note that an EB-map is necessarily a CP-map; in fact, they are PPT-maps. Various characterizations of EB-maps are known in the literature.

\begin{thm}[\cite{Hol99, HSR03}]\label{thm-EBmap-char}
 Let $\Phi: \M{d_1}\to \M{d_2}$ be a CP-map. Then the following conditions are equivalent: 
           \begin{enumerate}[label=(\roman*)]
                  \item $\Phi$ is entanglement breaking.
                  \item $\Gamma\circ\Phi$ is CP for all positive maps $\Gamma:\M{d_2}\to\M{n}, n\geq 1$.
                  \item $\Gamma\circ\Phi$ is CP for all positive maps $\Gamma:\M{d_2}\to\M{d_1}$.
                  \item $\Phi\circ\Gamma$ is CP for all positive maps $\Gamma:\M{n}\to\M{d_1}, n\geq 1$.
                  \item $\Phi\circ\Gamma$ is CP for all positive maps $\Gamma:\M{d_2}\to\M{d_1}$.
                  \item $C_\Phi\in(\M{d_1}\otimes\M{d_2})^+$ is separable.
                  \item $\Phi$ admits a set of rank one Choi-Kraus operators. 
                  \item (Holevo form:) There exist $F_i\in\M{d_1}^+$ and $R_i\in\M{d_2}^+$ such that $\Phi(X)=\sum_{i=1}^m\tr(XF_i)R_i$ for all $X\in\M{d_1}$.
           \end{enumerate} 
\end{thm}
 
  One may also consider CP-maps that have Choi-Kraus operators of rank less than or equal to some $k>1$. A  CP-map $\Phi:\M{d_1}\to\M{d_2}$ is said to be \emph{$k$-partially entanglement breaking} (k-PEB) if  it has Choi-Kraus operators of rank less than or equal to $k$ (equivalently $SN(\Phi)\leq k$).  Refer \cite{ChKo06} for details. EB-maps and $k$-PEB maps are also known as \emph{superpositive maps} and \emph{$k$-superpositive maps}, respectively, in the literature (See \cite{And04, SSZ09}).

\begin{defn}[\cite{Sto86,Sko11}]
 Let $\mcl{P}(d_1,d_2)$ denotes the cone of all positive linear maps from $\M{d_1}$ into $\M{d_2}$. A \emph{mapping cone} is a closed convex cone $\msc{C}\subseteq\mcl{P}(d_1,d_2)$ such that $$\Phi\circ\Theta\circ\Psi\in \msc{C}$$ for all $\Theta\in \msc{C}$, and CP-maps $\Psi:\M{d_1}\to\M{d_1}$ and $\Phi:\M{d_2}\to\M{d_2}$. The \emph{dual cone} of a mapping cone $\msc{C}$ is defined by  
 $$\msc{C}^\circ:=\{\Gamma\in\mcl{P}(d_1,d_2): \tr(C_\Gamma C_\Theta)\geq 0 \mbox{ for all }\Theta\in \msc{C}\}.$$
  A mapping cone is said to be \emph{invariant} if $\Gamma_2\circ\Theta\circ\Gamma_1\in \msc{C}$ for all $\Theta\in \msc{C}$ and $\Gamma_i\in\mcl{P}(d_i,d_i), i=1,2$. 
\end{defn} 

  If $\msc{C}$ is a mapping cone, then the dual cone $\msc{C}^\circ$ is also a mapping cone. Further, if $\msc{C}$ is invariant, then so is $\msc{C}^\circ$. The following are some well-known examples of  mapping cones of $\mcl{P}(d_1,d_2)$:  
 \begin{align*}
    \mcl{P}_k(d_1,d_2)&:=\{\mbox{$k$-positive maps from $\M{d_1}$ into $\M{d_2}$}\},\\
    \mcl{CP}(d_1,d_2)&:=\{\mbox{CP-maps from $\M{d_1}$ into $\M{d_2}$}\},\\
    \mcl{PEB}_k(d_1,d_2)&:=\{\mbox{$k$-PEB maps from $\M{d_1}$ into $\M{d_2}$}\},\\
    \mcl{EB}(d_1,d_2)&:=\{\mbox{EB-maps from $\M{d_1}$ into $\M{d_2}$}\},   
 \end{align*}
 where $k>1$. (If $d_1=d_2=d$, then we write $\mcl{P}_k(d),\mcl{CP}(d),\mcl{PEB}_k(d),\mcl{EB}(d)$, respectively.)  It is known that $(\msc{C}^\circ)^\circ=\msc{C}$ for every mapping cone $\msc{C}$, and  
 \begin{align}
      \mcl{P}(d_1,d_2)^\circ&= \mcl{EB}(d_1,d_2)\qquad(\mbox{hence }\mcl{EB}(d_1,d_2)^\circ=\mcl{P}(d_1,d_2)) \label{eq-P-dual}\\
      \mcl{P}_k(d_1,d_2)^\circ&=\mcl{PEB}_k(d_1,d_2)\quad(\mbox{hence }\mcl{PEB}_k(d_1,d_2)^\circ=\mcl{P}_k(d_1,d_2))\label{eq-P-n-dual}\\
      \mcl{CP}(d_1,d_2)^\circ&=\mcl{CP}(d_1,d_2). \label{eq-CP-dual}
 \end{align}
 The following much more generalized result is known. See  \cite{Sto09, SSZ09, Sto11, Sko11} for details.
 
 \begin{thm}\label{thm-dual-cone}
 Let $\msc{C}\subseteq\mcl{P}(d_1,d_2)$ be a mapping cone and $\Gamma\in\mcl{P}(d_1,d_2)$. Then the following conditions are equivalent:
 \begin{enumerate}[label=(\roman*)]
     \item $\Gamma\in \msc{C}^\circ$.
     \item $\Gamma\circ\Theta^*\in \mcl{CP}(d_2)$ for all $\Theta\in \msc{C}$.
     \item $\Theta^*\circ\Gamma\in\mcl{CP}(d_1)$ for all $\Theta\in \msc{C}$.
  \end{enumerate}
 Further, if $d_1=d_2=d$ and $\msc{C}$ is $\ast$-invariant, then the above conditions are equivalent to:  
 \begin{enumerate}
      \item [(iv)] $\Gamma\circ\Theta\in \mcl{CP}(d)$ for all $\Theta\in \msc{C}$.
     \item [(v)] $\Theta\circ\Gamma\in\mcl{CP}(d)$ for all $\Theta\in \msc{C}$.
 \end{enumerate}
\end{thm}

\section{Mapping cone of \texorpdfstring{$k$}{k}-entanglement breaking maps}\label{sec-nEB}

\begin{defn}[\cite{CMW19}]
 Let $k\in\mbb{N}$. A linear map $\Phi:\M{d_1}\to\M{d_2}$ is said to be \emph{$k$-entanglement breaking ($k$-EB)} if $\Phi$ is  $k$-positive map and $(\id_k\otimes\Phi)(X)$ is separable for every $X\in(\M{k}\otimes\M{d_1})^+$. 
\end{defn}

 Note that a  linear map $\Phi:\M{d_1}\to\M{d_2}$ is entanglement breaking if and only if $\Phi$ is $k$-entanglement breaking for every $k\geq 1$.  Further, from Theorem \ref{thm-EBmap-char}(vi), this is equivalent to saying that $\Phi$ is $d_1$-entanglement breaking. 
 
\begin{rmk}\label{rmk-nEB}
 Let $\Phi:\M{d_1}\to\M{d_2}$ be a $k$-EB map.   
 \begin{enumerate}[label=(\roman*)]
      \item If $m< k$, then $\Phi$ is $m$-EB also. For, let $X\in(\M{m}\otimes\M{d_1})^+=(\M{m}(\M{d_1}))^+$. Then there exist $A_i\in\M{k}^+$ and $B_i\in\M{d_2}^+$ such that 
 \begin{align*}
    \sMatrix{(\id_m\otimes\Phi)(X)&0\\0&0_{k-m}}
    =(\id_k\otimes\Phi)(\sMatrix{X&0\\0&0_{k-m}})
    =\sum A_i\otimes B_i.
 \end{align*}
 Writing $A_i=\sMatrix{A_{11}^{(i)}&A_{12}^{(i)}\\A_{21}^{(i)}&A_{22}^{(i)}}\in\M{k}^+$, where $A_{11}^{(i)}\in\M{m}^+$ and $A_{22}^{(i)}\in\M{k-m}^+$, from above we get
 $$(\id_m\otimes\Phi)(X)=\sum A_{11}^{(i)}\otimes B_i\in\M{m}^+\otimes\M{d_2}^+,$$
 so that $(\id_m\otimes\Phi)(X)$ is separable. 
    \item Let $m\geq 1$. Then, from the definition, it follows that both $\Gamma_2\circ\Phi$ and $\Phi\circ\Gamma_1$ are $k$-EB maps for all $\Gamma_2\in\mcl{P}(d_2,m)$ and $\Gamma_1\in\mcl{P}_k(m,d_1)$. 
 \end{enumerate}
\end{rmk} 
 
 In this section, our main aim is to prove an analog of Theorem \ref{thm-EBmap-char} in the context of $k$-EB maps. First, we prove some technical lemmas. Though it did not explicitly state, the following lemma was observed in \cite{CMW19}.
 
\begin{lem}\label{lem-SN-decomp}
 Let $X\in(M_m\otimes\M{d_1})^+$. Then for $j=1,2,\cdots, SN(X)$ there exist family of isometries $V_{j_i}:\mbb{C}^j\to\mbb{C}^{m}$ and vectors $\psi_{j_i}\in\mbb{C}^j\otimes\mbb{C}^{d_1}$ such that
 \begin{align*}
   X=\sum_{i,j}(V_{j_i}\otimes I_{d_1})\ranko{\psi_{j_i}}{\psi_{j_i}}(V_{j_i}\otimes I_{d_1})^*.
 \end{align*}
\end{lem} 
 
\begin{proof}
 Let $r:=SN(X)$. Then there exist $z_i\in\mbb{C}^{m}\otimes\mbb{C}^{d_1}$ with $SR(z_i)\leq r$ such that  
 \begin{align}\label{eq-X-decomp}
   X=\sum_{i}\ranko{z_i}{z_i}=\sum_{j=1}^r\sum_{SR(z_i)=j}\ranko{z_i}{z_i}.
 \end{align}
  Note that whenever $SR(z_i)=j$, from \cite[Lemma 1.2]{CMW19}, there exist isometry $V_{j_i}:\mbb{C}^j\to\mbb{C}^{m}$ and a vector $\psi_{j_i}\in\mbb{C}^j\otimes\mbb{C}^{d_1}$ such that $z_i=(V_{j_i}\otimes I_{d_1})\psi_{j_i}$. Thus 
 $$X=\sum_{j=1}^r\sum_{i}(V_{j_i}\otimes I_{d_1})\ranko{\psi_{j_i}}{\psi_{j_i}}(V_{j_i}\otimes I_{d_1})^*.$$
 In \eqref{eq-X-decomp}, if there is no $z_i$ with $SR(z_i)=j$, then the corresponding term is taken to be zero, and take $\psi_{j_i}=0$. 
\end{proof} 
  
 The following result is known in the quantum information theory (c.f \cite{TeHo00, RaAl07,SSZ09, Sko11}), but we could not find out in the following form (except $(i)\Leftrightarrow(v)$), hence providing proof.
 
\begin{lem}\label{lem-k-pos}
 Let $\Phi:\M{d_1}\to\M{d_2}$ be a positive map and $k\geq 1$. Then the following conditions are equivalent: 
 \begin{enumerate}[label=(\roman*)]
    \item $\Phi$ is $k$-positive.
    \item $(\id_{m}\otimes\Phi)(X)\in (\M{m}\otimes\M{d_2})^+$ for all $X\in(\M{m}\otimes\M{d_1})^+$ with $SN(X)\leq k$ and $m\geq 1$.
    \item $(\id_{d_2}\otimes\Phi)(X)\in (\M{d_2}\otimes\M{d_2})^+$ for all $X\in(\M{d_2}\otimes\M{d_1})^+$ with $SN(X)\leq k$.
    \item $\Phi\circ\Psi$ is CP for every $k$-PEB map $\Psi:\M{m}\to\M{d_1}$ and $m\geq 1$. 
    \item $\Phi\circ\Psi$ is CP for every $k$-PEB map $\Psi:\M{d_2}\to\M{d_1}$.
 \end{enumerate}
\end{lem} 
 
\begin{proof}
 The implications $(ii)\Rightarrow (i)$, $(ii)\Rightarrow (iii)$, $(iv)\Rightarrow (v)$ are trivial, and $(v)\Rightarrow (i)$ follows from \cite[Theorem 3]{Sko11}. Now we prove the remaining implications.\\ \\ 
  $(i)\Rightarrow (ii)$ Let $X\in(\M{m}\otimes\M{d_1})^+$. If $m\leq k$, then $(\id_{m}\otimes\Phi)(X)\in (\M{m}\otimes\M{d_2})^+$ as $\Phi$ is $m$-positive. Now suppose $m>k$ and  $r:=SN(X)\leq k$. Then , from Lemma \ref{lem-SN-decomp}, there exist isometries $V_{j_i}:\mbb{C}^j\to\mbb{C}^{m}$ and a vectors $\psi_{j_i}\in\mbb{C}^j\otimes\mbb{C}^{d_1}$, where $j\leq r$ such that  
 \begin{align}\label{eq-k-pos-decomp}
     (\id_{m}\otimes\Phi)(X)
         &=(\id_{m}\otimes\Phi)\Big(\sum_{i,j}(V_{j_i}\otimes I_{d_1})\ranko{\psi_{j_i}}{\psi_{j_i}}(V_{j_i}\otimes I_{d_1})^*\Big). \notag\\
         &=\sum_{i,j}(V_{j_i}\otimes I_{d_2})\big((\id_j\otimes\Phi)(\ranko{\psi_{j_i}}{\psi_{j_i}})\big)(V_{j_i}\otimes I_{d_2})^*.
 \end{align}
 Since $\Phi$ is $k$-positive we have $(\id_j\otimes\Phi)(\ranko{\psi_{j_i}}{\psi_{j_i}})$ is positive for all $i$ and $ j\leq r$. From the above equation we conclude that $(\id_{m}\otimes\Phi)X$ is positive.\\ \\
 $(ii)\Rightarrow (iv)$ Let $\Psi\in\mcl{PEB}_k(m,d_1)$, where $m\geq 1$. Then $C_\Psi\in (\M{m}\otimes\M{d_1})^+$ and $SN(C_\Psi)=SN(\Psi)\leq k$, hence by assumption $C_{\Phi\circ\Psi}=(\id_m\otimes\Phi)(C_\Psi)\in(\M{m}\otimes\M{d_2})^+$. Thus $\Phi\circ\Psi$ is a CP-map. \\ \\ 
 $(iv)\Rightarrow (ii)$ Let $m\geq 1$ and $X\in(\M{m}\otimes\M{d_1})^+$ be such that  $SN(X)\leq k$. Choose $\Psi\in\mcl{CP}(m,d_1)$ such that $C_\Psi=X$. Since $SN(\Psi)=SN(X)\leq k$, by assumption, $\Phi\circ\Psi$ is a CP-map, so that $(\id_{m}\otimes\Phi)(X)=C_{\Psi\circ\Psi}$ is positive. \\ \\
 $(iii)\Leftrightarrow (v)$ Proof is the same as that of $(ii)\Leftrightarrow (iv)$.  
\end{proof}   
  
\begin{thm}\label{thm-nEB-char}
 Let $\Phi:\M{d_1}\to\M{d_2}$ be a $k$-positive map, where $k> 1$. Then the following conditions are equivalent:
 \begin{enumerate}[label=(\roman*)]
       \item\label{nEB} $\Phi$ is $k$-entanglement breaking.
       \item\label{nEB-CP} $\Phi\circ\Psi$ is entanglement breaking for all CP-maps $\Psi:\M{k}\to\M{d_1}$.   
       \item\label{nEB-mCP} $\Phi\circ\Psi$ is entanglement breaking for all CP-maps $\Psi:\M{m}\to\M{d_1}$ and $1\leq m\leq k$.  
       \item\label{nEB-nPEB} $\Phi\circ\Psi$ is entanglement breaking for all $k$-PEB maps $\Psi:\M{d_1}\to\M{d_1}$.
       \item\label{nEB-nPEB-m} $\Phi\circ\Psi$ is entanglement breaking for all $k$-PEB maps $\Psi:\M{m}\to\M{d_1}$ and $m\geq 1$. 
       \item\label{nEB-proj} $\Phi\circ Ad_P$ is entanglement breaking for all projections $P\in\M{d_1}$ with rank$(P)=\min\{k,d_1\}$.
       \item\label{nEB-SN-m} $SN\big((\id_{m}\otimes\Phi)(X)\big)=1$ for all $X\in(\M{m}\otimes\M{d_1})^+$ with $SN(X)\leq k$ and $m\geq 1$.
       \item\label{nEB-SN} $SN\big((\id_{d_1}\otimes\Phi)(X)\big)=1$ for all $X\in(\M{d_1}\otimes\M{d_1})^+$ with $SN(X)\leq k$.
      \item\label{pos-nEB} $\Gamma\circ\Phi$ is $k$-positive for all positive maps $\Gamma:\M{d_2}\to\M{m}$ and $m\geq 1$.  
      \item\label{pos-nEB-1} $\Gamma\circ\Phi$ is $k$-positive for all positive maps $\Gamma:\M{d_2}\to\M{d_1}$.
      \item\label{nEB-tr} $\tr(C_\Phi C_{\Gamma\circ\Psi})\geq 0$ for all $k$-PEB maps $\Psi:\M{d_1}\to\M{d_1}$ and positive maps $\Gamma:\M{d_1}\to\M{d_2}$. 
    \end{enumerate}
\end{thm}  
  
\begin{proof}
 The implications $\ref{nEB-mCP}\Rightarrow\ref{nEB-CP}$,  $\ref{nEB-nPEB-m}\Rightarrow\ref{nEB-nPEB}$,  $\ref{nEB-SN-m}\Rightarrow\ref{nEB-SN}$, $\ref{nEB-nPEB}\Rightarrow\ref{nEB-proj}$ and   $\ref{pos-nEB}\Rightarrow\ref{pos-nEB-1}$ are trivial. Now we prove the remaining implications as follows:\\ \\
 $\ref{nEB}\Rightarrow\ref{nEB-CP}$ Let $\Psi\in\mcl{CP}(k,d_1)$. Since $C_\Psi\in (\M{k}\otimes\M{d_1})^+$ and $\Phi$ is $k$-EB $C_{\Phi\circ\Psi}=(\id_k\otimes\Phi)(C_\Psi)$ is separable, so that $\Phi\circ\Psi$ is an EB-map. \\ \\
 $\ref{nEB-CP}\Rightarrow\ref{nEB}$ Let $X\in(\M{k}\otimes\M{d_1})^+$ and $\Psi\in\mcl{CP}(k,d_1)$ be such that $C_\Psi=X$. By assumption $\Phi\circ\Psi$ is an EB-map, hence $(\id_k\otimes\Phi)(X)=C_{\Phi\circ\Psi}$ is separable. Since $X\in(\M{k}\otimes\M{d_1})^+$ is arbitrary this implies that $\Phi$ is a $k$-EB map.\\  \\
 $\ref{nEB}\Rightarrow\ref{nEB-mCP}$. Since $\Phi$ is $m$-EB for $1\leq m\leq k$, from the equivalence of $\ref{nEB}$ and $\ref{nEB-CP}$, we have $\Phi\circ\Psi$ is EB for all $\Psi\in\mcl{CP}(m,d_1)$.\\ \\
 $\ref{nEB}\Rightarrow\ref{nEB-SN-m}$ Let $X\in(\M{m}\otimes\M{d_1})^+$ with $r:=SN(X)\leq k$. Then, by  Lemma \ref{lem-k-pos},  $(\id_{m}\otimes\Phi)(X)$ is positive. Further, by Lemma \ref{lem-SN-decomp}, there exist isometries  $V_{j_i}:\mbb{C}^j\to\mbb{C}^{m}$ and vectors $\psi_{j_i}\in\mbb{C}^j\otimes\mbb{C}^{d_1}, 1\leq j\leq r$ such that 
  \begin{align*}
     SN\big((\id_{m}\otimes\Phi)(X)\big)
         &=SN\Big(\sum_{i,j}(V_{j_i}\otimes I_{d_2})\big((\id_j\otimes\Phi)(\ranko{\psi_{j_i}}{\psi_{j_i}})\big)(V_{j_i}\otimes I_{d_2})^*\Big)\\
         &\leq\max_{i,j}\{SN((V_{j_i}\otimes I_{d_2})\big((\id_j\otimes\Phi)(\ranko{\psi_{j_i}}{\psi_{j_i}})\big)(V_{j_i}\otimes I_{d_2})^*)\}\\
         &\leq \max_{i,j}\{SN\big((\id_j\otimes\Phi)\ranko{\psi_{j_i}}{\psi_{j_i}}\big)\}\\
         &=1. 
 \end{align*}
 The last equality follows as $\Phi$ is $j$-entanglement breaking for all $1\leq j\leq k$.\\ \\
 $\ref{nEB-SN-m}\Rightarrow\ref{nEB-nPEB-m}$
 Let $\Psi\in\mcl{PEB}_k(m,d_1)$, where $m\geq 1$.  Since $C_\Psi\in(\M{m}\otimes\M{d_1})^+$ and $SN(C_\Psi)\leq k$, by Lemma \ref{lem-k-pos},  $C_{\Phi\circ\Psi}=(\id_{m}\otimes\Phi)(C_\Psi)\in(\M{m}\otimes\M{d_2})^+$. 
 Further, from the assumption 
 \begin{align*}
    SN(C_{\Phi\circ\Psi})=SN\big((\id_{m}\otimes\Phi)C_\Psi\big)=1.
 \end{align*}
 Thus, $\Phi\circ\Psi$ is an EB-map. \\ \\
 $\ref{nEB-nPEB}\Rightarrow\ref{nEB}$ 
  We prove this in two case. \\
  {\ul{Case (1)}:} Suppose $k\geq d_1$. Then $\Psi=\id_{d_1}$ is a $k$-PEB map, hence by assumption $\Phi=\Phi\circ\Psi$ is an EB-map. In particular, $\Phi$ is a $k$-EB map.\\
 {\ul{Case (2)}:} Suppose $k< d_1$. Let $X\in(\M{k}\otimes\M{d_1})^{+}$. Without loss of generality assume that $X$ is of rank one; hence there exists $V\in\M{k\times d_1}$ such that $X=C_{Ad_{V}}$. Let $\iota:\mbb{C}^k\to\mbb{C}^{d_1}$ be the inclusion map. Note that  $SN(Ad_{\iota\circ V})=rank(\iota\circ V)\leq k$. Thus, $Ad_{\iota\circ V}\in\mcl{PEB}_k(d_1)$, hence by assumption  $C_{\Phi\circ Ad_{\iota\circ V}}$ is separable. Therefore, 
 \begin{align*}
      (\id_{k}\otimes\Phi)(X)
          =C_{\Phi\circ Ad_V}
          =(\iota\otimes I_{d_2})^*C_{\Phi\circ Ad_{\iota\circ V}}(\iota\otimes I_{d_2})
 \end{align*}
 is separable.  Since $X\in(\M{k}\otimes\M{d_1})^{+}$ is arbitrary we conclude that  $\Phi$ is $k$-EB. \\   \\
 $\ref{nEB-SN}\Rightarrow\ref{nEB-nPEB}$
  Let  $\Psi\in\mcl{PEB}_k(d_1)$. Since $SN(C_{\Psi})\leq k$, by assumption $C_{\Phi\circ \Psi} = (\id_{d_1}\otimes\Phi)(C_{\Psi})$ is separable. Hence $\Phi\circ\Psi$ is EB.\\ \\
  $\ref{nEB}\Rightarrow\ref{pos-nEB}$ Let $\Gamma\in\mcl{P}(d_2,m)$, where $m\geq 1$ and $X\in (\M{k}\otimes\M{d_1})^{+}$. By assumption $(\id_{k}\otimes\Phi)(X)$ is separable, hence there exist $A_{i}\in\M{k}^{+}$ and $B_{i}\in\M{d_2}^{+}$ such that $(\id_{k}\otimes\Phi)(X)=\sum_{i}A_{i}\otimes B_{i}.$ Then, $(\id_{k}\otimes\Gamma\circ\Phi)(X)=\sum_{i}A_{i}\otimes \Gamma(B_{i})\geq 0$. We conclude that $\Gamma\circ\Phi$ is $k$-positive.\\ \\
 $\ref{pos-nEB-1}\Rightarrow\ref{nEB-nPEB}$ Let $\Psi\in\mcl{PEB}(d_1)$. If $\Gamma\in\mcl{P}(d_2,d_1)$, then by assumption $\Gamma\circ\Phi\in\mcl{P}_k(d_1)$, and hence from \eqref{eq-P-n-dual}   and Theorem \ref{thm-dual-cone}, $\Gamma\circ\Phi\circ\Psi$ is CP. Since $\Gamma$ is arbitrary, by Theorem \ref{thm-EBmap-char},  $\Phi\circ\Psi$ is EB.\\ \\
 $\ref{nEB-proj}\Rightarrow\ref{nEB-nPEB}$ If $k\geq d_1$, then by assumption $\Phi=\Phi\circ Ad_{I_{d_1}}$ is EB so that $\Phi\circ\Psi$ is also EB for every $\Psi\in\mcl{PEB}_k(m,d_1)$ and $m\geq 1$. Now suppose $k<d_1$. We show that $\Phi\circ Ad_V$ is EB for all $V\in\M{d_1}$ with rank$(V)\leq k$.  By singular value decomposition there exists unitaries $U,U'\in\M{d_1}$ and a diagonal matrix $D\in\M{r}$ such that $V=U^*\sMatrix{D&0\\0&0}U'$, where $r=rank(V)$. Then $V=XP$, where $X=U^*\sMatrix{D&0\\0&0}U'$ and $P=U'^*\sMatrix{I_k&0\\0&0}U'$, which is a projection of rank $k$. By assumption $\Phi\circ Ad_P$ is EB, and hence $\Phi\circ Ad_V=\Phi\circ Ad_P\circ Ad_X$ is also EB.  Since $V\in\M{d_1}$ is arbitrary we have $\Phi\circ\Psi$ is EB for all $\Psi\in\mcl{PEB}_k(d_1)$.\\ \\
 $\ref{pos-nEB-1}\Rightarrow\ref{nEB-tr}$ Let $\Gamma\in\mcl{P}(d_1,d_2)$ and $\Psi\in\mcl{PEB}_k(d_1)$. By assumption, $\Gamma^*\circ\Phi$ is $k$-positive. Hence from \eqref{eq-P-n-dual}, we have $0\leq\tr(C_{\Gamma^*\circ\Phi}C_\Psi)=\tr(C_{\Phi}C_{\Gamma\circ\Psi})$. \\ \\
 $\ref{nEB-tr}\Rightarrow\ref{pos-nEB-1}$  Let $\Gamma\in\mcl{P}(d_2,d_1)$.  Then by assumption $\tr(C_{\Gamma\circ\Phi}C_\Psi)=\tr(C_{\Phi}C_{\Gamma^*\circ\Psi})\geq 0$ for all $\Psi\in\mcl{PEB}_k(d_1)$. Now from \eqref{eq-P-n-dual} it follows that $\Gamma\circ\Phi$ is $k$-positive.
\end{proof}  

\begin{note}
 From \ref{nEB-nPEB} it follows that if $\Phi$ is $k$-EB, then $\Phi\circ Ad_P$ is entanglement breaking for all projections $P\in\M{d_1}$ with rank less than or equal to $k$. The converse is also true because of \ref{nEB-proj}. This equivalence was observed in \cite[Lemma 6.1]{ChCh20}.   
\end{note}

  From Theorem \ref{thm-EBmap-char}, we have $\Phi:\M{d_1}\to\M{d_2}$ is an EB-map if and only if the Choi matrix $C_\Phi=\sum_{i,j}E_{ij}\otimes \Phi(E_{ij})$ is separable, where $\{E_{ij}\}_{i,j=1}^{d_1}$ is the standard matrix units in $\M{d_1}$. Next we prove an analogue of this result in the context of $k$-EB maps. Suppose $\{F_{ij}\}_{i,j=1}^{d}$ is a complete set of matrix units in $\M{d_1}$ (i.e., there exists an orthonormal basis $\{f_i\}_{i=1}^d$ for $\mbb{C}^d$ such that $F_{ij}=\ranko{f_i}{f_j}$) and $U\in\M{d_1}$ is a unitary such that $UF_{ij}U^*=E_{ij}$, for all $1\leq i,j\leq d_1$. Then, by \cite[Lemma 4.1.2]{Sto13}, the Choi matrix $C_\Phi^F$ w.r.t $\{F_{ij}\}_{i,j}$ is given by  
  \begin{align}\label{eq-Stormer}
      C_\Phi^F:=\sum_{i,j=1}^{d_1}F_{ij}\otimes\Phi(F_{ij})=Ad_{U\otimes I_{d_2}}C_{\Phi\circ Ad_U}.
  \end{align}   

\begin{defn} 
Suppose $\Phi:\M{d_1}\to\M{d_2}$ is a linear map and $\{F_{ij}\}_{i,j=1}^{d_1}$ is a complete set of matrix units in $\M{d_1}$. Given $1\leq k\leq d_1$ and $i_1,i_2,\cdots, i_k\in\{1,2,\cdots,d_1\}$ the matrix  
 \begin{align*}
  C_{\Phi(i_1,\cdots,i_k)}^F
   :=\sum_{p,q=1}^k E_{pq}^{(k)}\otimes \Phi(F_{{i_p}i_q})
    =\Matrix{\Phi(F_{i_1i_1})&\Phi(F_{i_1i_2})&\cdots&\Phi(F_{i_1i_k})\\ 
                \vdots&\vdots&\cdots&\vdots\\
                \Phi(F_{i_ki_1})&\Phi(F_{i_ki_2})&\cdots&\Phi(F_{i_ki_k}) 
               }
 \end{align*} 
  in $\M{k}\otimes\M{d_2}=\M{k}(\M{d_2})$ is called a \emph{principal $(k\times k)$-block submatrix} of $C_\Phi^F$, where $E_{ij}^{(k)}\in\M{k}$ is the standard matrix units. The principal $(k\times k)$-block submatrix  $C^E_{\Phi(i_1,\cdots,i_k)}$ of $C_\Phi$ w.r.t the standard matrix units $\{E_{ij}\}_{i,j}\subseteq\M{d_1}$ is denoted simply by $C_{\Phi(i_1,\cdots,i_k)}$.
\end{defn}
 
  Let $\{e_i^{(k)}\}_{i=1}^k$ be the standard orthonormal basis for $\mbb{C}^k$ and $\{f_i\}_{i=1}^{d_1}$ be an orthonormal basis for $\mbb{C}^{d_1}$ such that $F_{ij}=\ranko{f_i}{f_j}$. Then 
  \begin{align*}
     C_{\Phi(i_1,\cdots,i_k)}^F
     =(\id_k\otimes\Phi)(\sum_{p,q=1}^kE_{pq}^{(k)}\otimes F_{i_pi_q})
     =(\id_k\otimes\Phi)(\ranko{\sum_{p=1}^k e_p^{(k)}\otimes f_{i_p}}{\sum_{q=1}^k e_q^{(k)}\otimes f_{i_q}}).
  \end{align*}
  Note that $\ranko{\sum_{p=1}^k e_p^{(k)}\otimes f_{i_p}}{\sum_{q=1}^k e_q^{(k)}\otimes f_{i_q}}\in (\M{k}\otimes\M{d_1})^+$. Hence, $C_{\Phi(i_1,\cdots,i_k)}^F$ is positive (resp. separable) if $\Phi$ is $k$-positive (resp. $k$-EB). 
 
\begin{thm}\label{thm-kEB-Choi-block}
 Let $\Phi:\M{d_1}\to\M{d_2}$ be a $k$-positive map, where $1\leq k\leq d_1$. Then the following conditions are equivalent:
 \begin{enumerate}[label=(\roman*)]
    \item $\Phi$ is $k$-EB.
    \item Given a complete set of matrix units $\{F_{ij}\}_{i,j=1}^{d_1}\subseteq\M{d_1}$, every  principal $(k\times k)$-block submatrix $C_{\Phi(i_1,\cdots,i_k)}^F$ of the Choi matrix $C_\Phi^F$  is separable. 
     \item Given a complete set of matrix units $\{F_{ij}\}_{i,j=1}^{d_1}\subseteq\M{d_1}$, the principal $(k\times k)$-block submatrix  $C_{\Phi(1,2,\cdots,k)}^F$ of the Choi matrix $C_\Phi^F$  is separable. 
 \end{enumerate}
\end{thm} 
 
\begin{proof}
 We need to prove $(iii)\Rightarrow(i)$ only. So let $P\in\M{d_1}$ be a projection of rank $k$ and $\wtilde{\Phi}=\Phi\circ Ad_P$. Suppose $\{f_i\}_{i=1}^{d_1}\subseteq\mbb{C}^{d_1}$ is an orthonormal basis such that $P=\sum_{i=1}^k\ranko{f_i}{f_i}$. Let $F_{ij}=\ranko{f_i}{f_j}$ and $U\in\M{d_1}$ be a  unitary such that $UF_{ij}U^*=E_{ij}\in\M{d_1}$ for all $1\leq i,j\leq d_1$. Since $\{F_{ij}\}_{ij}$ is a  complete set of matrix units, from \eqref{eq-Stormer}, we have 
 \begin{align*}
    C_{\wtilde{\Phi}\circ Ad_U}
       =Ad_{U^*\otimes I_{d_2}}C^F_{\wtilde{\Phi}}
       =Ad_{U^*\otimes I_{d_2}}\Big(\sum_{i,j=1}^{d_1} F_{ij}\otimes\wtilde{\Phi}(F_{ij})\Big)
       =Ad_{U^*\otimes I_{d_2}}\Big(\sum_{i,j=1}^{k} F_{ij}\otimes\Phi(F_{ij})\Big).
 \end{align*}
 Now let $W\in\M{d_1\times k}$ be such that $W(e_i^{(k)})=f_i$ for all $1\leq i\leq k$. Then  
 \begin{align*}
    C_{\wtilde{\Phi}\circ Ad_U}
       &=Ad_{U^*\otimes I_{d_2}}(W\otimes I_{d_2})\Big(\sum_{i,j=1}^{k} E_{ij}^{(k)}\otimes\Phi(F_{ij})\Big)(W\otimes I_{d_2})^*\\
       &=(UW\otimes I_{d_2})C^F_{\Phi(1,\cdots,k)}(UW\otimes I_{d_2})^*.
 \end{align*}
 Since $C^F_{\Phi(1,\cdots,k)}$ is separable we conclude that $C_{\wtilde{\Phi}\circ Ad_U}$ is separable. Hence $\Phi\circ Ad_P=\wtilde{\Phi}\circ Ad_U\circ Ad_{U^*}$ is an EB-map. Further, since $P$ is arbitrary, by Theorem \ref{thm-nEB-char}, $\Phi$ is $k$-EB. 
\end{proof} 

\begin{rmk}\label{rmk-nEB-T}
 If $\Phi:\M{d_1}\to\M{d_2}$ is a $k$-EB map, where $k\geq 1$, then $\T\circ\Phi\circ\T$ is also a $k$-EB map. For, let $\Psi\in\mcl{CP}(k,d_1)$. Since $\T\circ\Psi\circ\T\in\mcl{CP}(k,d_1)$, by Theorem \ref{thm-nEB-char}, $\Phi\circ(\T\circ\Psi\circ\T)$ is EB. Hence  
 $$\T\circ\Phi\circ\T\circ\Psi=\T\circ(\Phi\circ\T\circ\Psi\circ\T)\circ\T$$
 is also an EB-map.  Since $\Psi$ is arbitrary, by Theorem \ref{thm-nEB-char},  $\T\circ\Phi\circ\T$ is $k$-EB.  
\end{rmk}

\begin{thm}\label{thm-nEB-flipped}
 Let $\Phi:\M{d_1}\to\M{d_2}$ be a $k$-positive map, where $k\geq 1$. Then the following conditions are equivalent: 
  \begin{enumerate}[label=(\roman*)]
      \item $\Phi$ is $k$-entanglement breaking.
      \item $(\Phi\otimes\id_k)(X)$ is separable for all $X\in(\M{d_1}\otimes\M{k})^+$.  
  \end{enumerate}
\end{thm}
 
\begin{proof} 
 $(i)\Rightarrow(ii)$ Let $X\in(\M{d_1}\otimes\M{k})^+$, and without loss of generality assume that  $X=\ranko{\psi}{\psi}$ for some $\psi\in\mbb{C}^{d_1}\otimes\mbb{C}^k$.  By \cite[Lemma 1.3]{CMW19}, there exists $V\in\M{k\times d_1}$ such that $\psi=(I_{d_1}\otimes V)(\Omega_{d_1})$. Further, applying \cite[Lemma 1.1]{CMW19} to the $k$-positive map $\Phi$, we get   
      \begin{align*}
         (\Phi\otimes\id_k)(X)
            &=(\Phi\otimes\id_k)\big((I_{d_1}\otimes V)(\ranko{\Omega_{d_1}}{\Omega_{d_1}})(I_{d_1}\otimes V)^*\big)\\
            &=(\Phi\otimes\id_k)(\id_{d_1}\otimes Ad_{V^*})(\ranko{\Omega_{d_1}}{\Omega_{d_1}})\\
            &=(\id_{d_2}\otimes Ad_{V^*})(\Phi\otimes\id_{d_1})(\ranko{\Omega_{d_1}}{\Omega_{d_1}})\\    
           &=(\id_{d_2}\otimes Ad_{V^*})(\id_{d_2}\otimes\T\circ\Phi^*\circ\T)(\ranko{\Omega_{d_2}}{\Omega_{d_2}})\\         
           &=\big(\id_{d_2}\otimes (\T\circ\Phi\circ\T\circ Ad_{V})^*\big)(\ranko{\Omega_{d_2}}{\Omega_{d_2}}).
      \end{align*}   
  Since $\T\circ\Phi\circ\T$ is $k$-EB, $\T\circ\Phi\circ\T\circ Ad_{V}:\M{k}\to\M{d_2}$ is EB. It follows from  the  above equation that $(\Phi\otimes\id_k)(X)$ is separable.   \\
  $(ii)\Rightarrow(i)$ Let $V\in\M{k\times d_1}$. Then $X=(\id_{d_1}\otimes Ad_{V^*})(\ranko{\Omega_{d_1}}{\Omega_{d_1}})\in(\M{d_1}\otimes\M{k})^+$, and the Choi-matrix of $Ad_{V^*}\circ\T\circ\Phi^*\circ\T:\M{d_2}\to\M{k}$ is given by
  \begin{align*}
     (\id_{d_2}\otimes Ad_{V^*}\circ\T\circ\Phi^*\circ\T)(\ranko{\Omega_{d_2}}{\Omega_{d_2}})
     =(\Phi\otimes\id_k)(\id_{d_1}\otimes Ad_{V^*})(\ranko{\Omega_{d_1}}{\Omega_{d_1}})
     =(\Phi\otimes\id_k)(X),
  \end{align*}
  which is separable by assumption. Thus $Ad_{V^*}\circ\T\circ\Phi^*\circ\T$ (and hence $\T\circ\Phi\circ\T\circ Ad_V$) is EB. Since $V$ is arbitrary, by Theorem \ref{thm-nEB-char}, $\T\circ\Phi\circ\T$ is $k$-EB.  Now, from Remark \ref{rmk-nEB-T}, we conclude that $\Phi$ is $k$-EB.
\end{proof}

\begin{note}
  Though the tensor product of CP-maps is a CP-map, it is not true in general that the tensor product of positive maps is a positive map. However, if $\Phi:\M{d_1}\to\M{d_2}$ is a $k$-EB map and $\Gamma:\M{k}\to\M{d_3}$ is any positive map, then from the definition  of  $k$-EB maps and Theorem \ref{thm-nEB-flipped}, it follows that
  \begin{align*}
     (\Phi\otimes\Gamma)(X)=(\id_{d_2}\otimes\Gamma)(\Phi\otimes\id_k)(X)
     \quad\mbox{and}\quad
     (\Gamma\otimes\Phi)(Y)=(\Gamma\otimes\id_{d_2})(\id_k\otimes\Phi)(Y)
  \end{align*}
  are separable for all $X\in (\M{d_1}\otimes\M{k})^+$ and  $Y\in (\M{k}\otimes\M{d_1})^+$. In particular, $\Phi\otimes\Gamma$ and $\Gamma\otimes\Phi$ are positive maps. 
\end{note} 

\begin{thm}\label{thm-mapp-cone}
 Let  $\mcl{EB}_k(d_1,d_2)$ and $\mcl{EBCP}_k(d_1,d_2)$ be the set of all $k$-entanglement breaking linear maps and $k$-entanglement breaking CP-maps from $\M{d_1}$ into $\M{d_2}$, respectively. Then both $\mcl{EB}_k(d_1,d_2)$ and $\mcl{EBCP}_k(d_1,d_2)$ are mapping cones.
\end{thm}

\begin{proof}
 Suppose $\{\Theta_n\}_{n=1}^\infty$ is a sequence in $\mcl{EB}_k(d_1,d_2)$ such that $\Theta_n\to\Theta\in\mcl{P}(d_1,d_2)$. Given any $\Psi\in\mcl{CP}(k,d_1)$, by Theorem \ref{thm-nEB-char}, $\Theta_n\circ\Psi$ is EB for every $n\geq 1$, and hence so is $\Theta\circ\Psi=\lim_n\Theta_n\circ\Psi$. Thus $\Theta\in\mcl{EB}_k(d_1,d_2)$, and concludes that $\mcl{EB}_k(d_1,d_2)$ is closed. Further, from Remark \ref{rmk-nEB}, it follows that $\mcl{EB}_k(d_1,d_2)$ is a mapping cone.  Now, since $\mcl{CP}_k(d_1,d_2)$ is a mapping cone the intersection
 $$\mcl{EB}_k(d_1,d_2)\bigcap\mcl{CP}_k(d_1,d_2)=\mcl{EBCP}_k(d_1,d_2)$$ is also a mapping cone.
\end{proof}

 The smallest closed convex cone containing a set $\mcl{S}$ will be denoted by $conv(\mcl{S})$.
 
\begin{thm}\label{thm-nEB-dual-cone}
 Let $k\geq 1$. Then\footnote{The identity  \eqref{EB-k-dual-2} was proved by Alexander M\"{u}ller-Hermes and communicated to us privately after we upload the preprint in arXiv. This proof leads us to prove the identity \eqref{EB-k-dual-1}.}
 \begin{align}
    \mcl{EB}_k(d_1,d_2)^\circ
      &=conv\{\Gamma\circ\Psi: \Gamma\in\mcl{P}(d_1,d_2), \Psi\in\mcl{PEB}_k(d_1)\} \label{EB-k-dual-1}\\
      &=conv\{\Gamma\circ\Psi: \Gamma\in\mcl{P}(k,d_2), \Psi\in\mcl{CP}(d_1,k)\}.      \label{EB-k-dual-2}                                 
 \end{align}
\end{thm}

\begin{proof}
 Let $\mcl{S}=\{\Gamma\circ\Psi: \Gamma\in\mcl{P}(d_1,d_2), \Psi\in\mcl{PEB}_k(d_1)\}$. By Theorem \ref{thm-nEB-char}\ref{nEB-tr}, we have $\mcl{S}\subseteq\mcl{EB}_k(d_1,d_2)^\circ$, and hence $conv(\mcl{S})\subseteq\mcl{EB}_k(d_1,d_2)^\circ$. Now to prove the reverse inclusion, it is enough to prove that $(conv(\mcl{S}))^\circ\subseteq(\mcl{EB}_k(d_1,d_2))^{\circ\circ}=\mcl{EB}_k(d_1,d_2)$. So let $\Phi\in(conv(\mcl{S}))^\circ$. Then   
 \begin{align*}
   \tr(C_{\Phi\circ\Psi} C_\Gamma)=\tr(C_\Phi C_{\Gamma\circ\Psi^*})\geq 0
 \end{align*}
 for all $\Gamma\in\mcl{P}(d_1,d_2)$ and $ \Psi\in\mcl{PEB}_k(d_1)$, hence $\Phi\circ\Psi\in\mcl{EB}(d_1,d_2)$. Since $\Psi$ is arbitrary, by Theorem \ref{thm-nEB-char}\ref{nEB-nPEB}, $\Phi\in\mcl{EB}_k(d_1,d_2)$. This completes the proof of \eqref{EB-k-dual-1}. Now to prove \eqref{EB-k-dual-2}, let 
 $$\mcl{C}_k(d_1,d_2):=conv\{\Gamma\circ\Psi: \Gamma\in\mcl{P}(k,d_2), \Psi\in\mcl{CP}(d_1,k)\}.$$
 To show that $\mcl{C}_k(d_1,d_2)\subseteq\mcl{EB}_k(d_1,d_2)^\circ$ note that
  \begin{align*}
     \tr(C_{\Gamma\circ\Psi}C_\Phi)=\tr\big(C_\Gamma(\id_k\otimes\Phi)(C_{\Psi^*})\big)\geq 0
  \end{align*} 
 for all $\Phi\in\mcl{EB}_k(d_1,d_2),\Gamma\in\mcl{P}(k,d_2)$ and $\Psi\in\mcl{CP}(d_1,k)$.  To show the reverse inequality it is enough to show that $\mcl{C}_k(d_1,d_2)^\circ\subseteq\mcl{EB}_k(d_1,d_2)$. To see this, consider $\Phi\in\mcl{C}_k(d_1,d_2)^\circ$. By definition, we know that 
 \begin{align*}
    \tr(C_{\Phi\circ\Psi} C_\Gamma)=\tr(C_\Phi C_{\Gamma\circ\Psi^*})\geq 0
 \end{align*} 
 for all $\Gamma\in\mcl{P}(k,d_2)$ and $\Psi\in\mcl{CP}(k,d_1)$. We conclude that $\Phi\circ\Psi\in\mcl{EB}(k,d_2)$ for every $\Psi\in\mcl{CP}(k,d_1)$, which is equivalent to $\Phi\in\mcl{EB}_k(d_1,d_2)$.
\end{proof}

\begin{rmk}
 Now, using \cite[Corollary 4]{Sto18}, we can compute the dual of $\mcl{EBCP}_k$ as $\mcl{EBCP}_k^\circ=(\mcl{EB}_k\cap\mcl{CP})^\circ=\mcl{EB}_k^\circ\vee\mcl{CP}^\circ$. Here, $\msc{C}_1\vee\msc{C}_2$ denotes the smallest closed convex cone containing the union of two mapping cones $\msc{C}_1$ and $\msc{C}_2$. Note that $\msc{C}_1\vee\msc{C}_2=\ol{\msc{C}_1+\msc{C}_2}$, where $\msc{C}_1+\msc{C}_2:=\{\Theta_1+\Theta_2 : \Theta_i\in\msc{C}_i,i=1,2\}$.   
\end{rmk}

 From Theorem \ref{thm-dual-cone} and Theorem  \ref{thm-nEB-char}\ref{nEB-nPEB-m}, it follows that $\mcl{PEB}_k(d_1,d_2)\subseteq\mcl{EB}_k(d_1,d_2)^\circ$ Thus we have the following: 
 \begin{align*}
 \begin{matrix}
    \mcl{C}^{\circ\circ}=\mcl{C}:&\mcl{P}(d_1,d_2)&\supseteq&\mcl{P}_k(d_1,d_2)&\supseteq&\mcl{EB}_k(d_1,d_2)&\supseteq&\mcl{EB}(d_1,d_2)\\
    &\Big\updownarrow&&\Big\updownarrow&&\Big\updownarrow&&\Big\updownarrow\\
    \mcl{C}^{\circ}:&\mcl{EB}(d_1,d_2)&\subseteq&\mcl{PEB}_k(d_1,d_2)&\subseteq&(\mcl{EB}_k(d_1,d_2))^\circ&\subseteq&\mcl{P}(d_1,d_2).
 \end{matrix}
\end{align*} 

\mbox{}\\
 
 We denote the set of all PPT-maps from $\M{d_1}$ into $\M{d_2}$ by $\mcl{PPT}(d_1,d_2)$. From Holevo form, it follows that every EB-map is a PPT-map. But the converse is not true in general.   In \cite{Hor97}, Horodecki proved that there exists $\Psi\in\mcl{PPT}(2,4)$ which is not EB (equivalently, not $2$-EB).   Now, given $d>4$ let $W=\Matrix{I_4&\large{\text{0}_{4\times(d-4)}}}\in\M{4\times d}$. Then the map $\Psi':=Ad_W\circ\Psi:\M{2}\to\M{d}$ is PPT but not EB since $Ad_{W*}\circ \Psi'=\Psi $ is not EB. Thus, for every $d\geq 4$ there is a PPT-map $\Psi':\M{2}\to\M{d}$ that is not EB.

\begin{prop}\label{prop-PPT-non2EB}
 Let $d\geq 4$. Then there exists a PPT map $\Phi:\M{d}\to\M{d}$ which is not $2$-EB. Moreover, we can choose $\Phi$ such that $SN(\Phi)=SN(\Phi\circ\T)=2$. 
\end{prop}

\begin{proof}
 Choose $\Psi\in\mcl{PPT}(2,d)$ that is not EB. Let $V=\Matrix{I_2\\ \large{\text{0}}_{(d-2)\times 2}}\in\M{d\times 2}$. Then the map $\Phi:=\Psi\circ Ad_V$ defined on $\M{d}$ is PPT, and both $SN(\Phi)$ and $SN(\Phi\circ\T)=SN(\T\circ\Phi)$ are less than or equal to $SN(Ad_V)=2$. Since $\Phi\circ Ad_{V*}=\Psi$ is not an EB-map, by Theorem \ref{thm-nEB-char}, $\Phi$ is not $2$-EB. Observe that $\Phi$ and $\Phi\circ\T$ has Schmidt number strictly greater than one as they are not EB. This completes the proof.
\end{proof}

 From the above discussion, we saw that for all $d\geq 4$ there exists a $\Psi\in\mcl{PPT}(d,2)$ which is not EB. The following proposition, which is a  generalization of  \cite[Theorem 3]{HHH96}, says that such a map will always be $3$-EB. 
 
\begin{prop}\label{prop-PPT-nEB}
 Let $d\geq 1$.
 \begin{enumerate}[label=(\roman*)]
    \item If $\Phi:\M{d}\to\M{2}$ is a PPT-map, then $\Phi$ is $3$-EB.
    \item If  $\Phi:\M{d}\to\M{3}$ is a PPT-map, then $\Phi$ is $2$-EB. 
 \end{enumerate}
\end{prop} 

\begin{proof}
 $(i)$ Let $\Psi\in\mcl{CP}(3,d)$. Then $\Phi\circ\Psi\in\mcl{PPT}(3,2)$, and hence by \cite[Theorem 3]{HHH96}, it is EB. Since $\Psi$ is arbitrary, from Theorem \ref{thm-nEB-char}, it follows that $\Phi$ is $3$-EB.\\
 $(ii)$ Follows from the fact that $\Phi\circ\Psi\in\mcl{PPT}(2,3)$, and hence EB for every $\Psi\in\mcl{CP}(2,d)$. 
\end{proof}

 A mapping cone $\msc{C}\subseteq\mcl{P}(d):=\mcl{P}(d,d)$ is said to be \emph{symmetric} (\cite{Sto11}) if $\Theta^*,\T\circ\Theta\circ\T\in \msc{C}$ for all $\Theta\in \msc{C}$. If $\msc{C}$ is a symmetric mapping cone, then $\msc{C}^\circ$ is also symmetric. The mapping cones  $\mcl{P}(d),\mcl{P}_k(d),\mcl{CP}(d),\mcl{PEB}_k(d),\mcl{EB}(d)$, where $k>1$, are known to be  symmetric. The following theorem says that, in general, the mapping cone $\mcl{EB}_k(d)$  (and hence $\mcl{EBCP}_k(d)$) is  not symmetric for $d\geq 4$.  

\begin{thm}\label{thm-kEB-nonsym}
 Let $d\geq 4$.
 \begin{enumerate}[label=(\roman*)]
    \item There exists a $3$-EB map $\Phi:\M{d}\to\M{2}$ such that $\Phi^*$ is not $2$-EB (and hence not $3$-EB).
    \item There exists a $3$-EB map $\Phi:\M{d}\to\M{d}$ such that $\Phi^*$ is not $2$-EB (and hence not $3$-EB).
 \end{enumerate}
\end{thm}

\begin{proof}
 $(i)$ Let $\Phi:\M{d}\to\M{2}$ be a linear map such that $\Phi^*$ is PPT-map but not EB.  Since $\Phi$ is PPT, by  Proposition \ref{prop-PPT-nEB},  $\Phi$ is $3$-EB.   But $\Phi^*:\M{2}\to\M{d}$ is not $2$-EB (equivalently not EB).\\
 $(ii)$ Let $\Phi'\in\mcl{EB}_3(d,2)$ be such that $\Phi'^*$ is not $2$-EB.  Take $\Phi=\iota\circ\Phi'$, where $\iota:\M{2}\to\M{d}$ is the inclusion map. Observe that $\Phi\in\mcl{EB}_3(d)$. Now, if $\Phi^*$ is $2$-EB, then $\Phi^*\circ \iota=\Phi'^*\circ\iota^*\circ\iota=\Phi'^*$ is also $2$-EB, which is not possible. Hence $\Phi^*$ is not $2$-EB.
\end{proof}

\begin{rmk}
 From Remark \ref{rmk-nEB} it follows that $\mcl{EB}_k(d_1,d_2)$ is a left-invariant mapping cone in the sense that it is invariant under composition by positive  maps  from the  left side. However,   $\mcl{EB}_k(d_1,d_2)$ is not right-invariant in general. If $\mcl{EB}_k(d_1,d_2)$ is right-invariant, then from Theorem \ref{thm-nEB-char} \ref{pos-nEB} it follows that the  adjoint of a $k$-EB map is always $k$-EB which is not true in general, as seen in the above theorem.
\end{rmk}

\begin{rmk}
 We know that every  $d_1$-EB map $\Phi:\M{d_1}\to\M{d_2}$ is EB. However, from the above Theorem, we observe that a $d_2$-EB map $\Phi:\M{d_1}\to\M{d_2}$ is not necessarily  an entanglement breaking map.
\end{rmk}

Let $\msc{C},\msc{C}_1,\msc{C}_2$ be mapping cones of $\mcl{P}(d_1,d_2)$. If $\msc{C}$ is any of $\mcl{P}, \mcl{P}_k, \mcl{CP}, \mcl{PEB}_k, \mcl{EB}$, then $$\msc{C}\circ \T:=\{\Phi\circ \T : \Phi\in \msc{C}\}$$ is a mapping cone.  A mapping cone arises from $\mcl{P}, \mcl{P}_k, \mcl{CP}, \mcl{PEB}_k, \mcl{EB}$, via the operations $\msc{C}\mapsto \msc{C}\circ \T, (\msc{C}_1,\msc{C}_2)\mapsto \msc{C}_1\cap \msc{C}_2$, or $(\msc{C}_1,\msc{C}_2) \mapsto \msc{C}_1\vee \msc{C}_2$ is called \emph{typical}, otherwise called \emph{untypical} (\cite{Sko11}).   In \cite{JSS13}, the authors discussed  examples and properties of untypical mapping cones on $\M{d}$, emphasizing the case $d=2$. Observe that typical mapping cones are symmetric. By Theorem \ref{thm-kEB-nonsym}, $\mcl{EB}_k(d)$ and $\mcl{EBCP}_k(d), k=2,3$ are not symmetric for $d\geq 4$, hence not typical. Actually much more is true, as the following theorem shows. 

\begin{thm}\label{thm-nEB-untypical}
Let $d\geq 4$ and $1<k<d$. Then the mapping cones $\mcl{EB}_k(d)$  and $\mcl{EBCP}_k(d)$ are not typical. 
\end{thm}

\begin{proof}
 By Proposition \ref{prop-PPT-non2EB}, there exists a map $\Phi\in\mcl{PEB}_2(d)\cap\big(\mcl{PEB}_2(d)\circ\T\big)$ which is not $2$-EB, hence $\Phi\notin\mcl{EB}_k(d)$ and $\Phi\notin\mcl{EBCP}_k(d)$. Now, from \cite[Lemma 12]{JSS13}, it follows that $\mcl{EB}_k(d)$ and $\mcl{EBCP}_k(d)$ are not typical.
\end{proof} 
 
 Recall that a positive map is said to be \emph{decomposable} if it can be written as a sum of a CP-map and a co-CP-map. Characterization of  CP-maps and PPT-maps in terms of decomposable maps is known in the literature. See \cite[Theorem]{PiMo07}, \cite[Proposition 8]{Sto18}. Next, we establish a similar characterization for entanglement breaking maps.  

\begin{thm}
 Let $\Phi:\M{d_1}\to\M{d_2}$ be a $k$-entanglement breaking map, where $d_1,k> 1$. Then the following conditions are equivalent: 
 \begin{enumerate}[label=(\roman*)]
    \item $\Phi$ is entanglement breaking.
    \item $\id_n\otimes\Gamma\circ\Phi$ is decomposable  for all positive maps $\Gamma:\M{d_2}\to\M{m}$, where $m,n\geq 1$.
    \item $\id_k\otimes\Gamma\circ\Phi$ is decomposable for all positive maps $\Gamma:\M{d_2}\to\M{d_1}$. 
 \end{enumerate} 
\end{thm}

\begin{proof} 
 $(i)\Rightarrow(ii)$ Let $m\geq 1$ and $\Gamma\in\mcl{P}(d_2,m)$. Since $\Phi$ is entanglement breaking $\Gamma\circ\Phi$ is CP, and  hence $\id_n\otimes \Gamma\circ\Phi$ is decomposable $-$ it is CP itself for all $n\geq 1$. \\
 $(ii)\Rightarrow(iii)$ Clear.\\
 $(iii)\Rightarrow(i)$ Suppose $\Phi$ is not an EB-map. Then there exists a $\Gamma\in\mcl{P}(d_2,d_1)$ such that $\Gamma\circ\Phi$ is not CP. Since $\Phi$ is $k$-EB, by Theorem \ref{thm-nEB-char}, $\Gamma\circ\Phi\in\mcl{P}_k(d_1)$. Since $\Gamma\circ\Phi\notin\mcl{CP}(d_1)$,  by \cite[Theorem]{PiMo07}, we have $\id_k\otimes\Gamma\circ\Phi$ is not decomposable, which is a contradiction. This completes the proof. 
\end{proof}

\section{Examples}\label{sec-eg}

 In this section, we discuss some examples of $k$-EB maps. Recall that the Schur product of two matrices $A=[a_{ij}],B=[b_{ij}]\in\M{d}$ is denoted and defined by $A\circ B:=[a_{ij}b_{ij}]\in\M{d}$. Define the map $S_A:\M{d}\to\M{d}$ by $S_A(X)=A\circ X$ for all $X\in\M{d}$. It is well known (c.f. \cite{Pau02}) that $S_A$ is CP if and only if $S_A$ is positive if and only if $A$ is positive. Further, $S_A$ is PPT if and only if $S_A$ is EB if and only if $A$ is positive diagonal. See \cite{RJP18,KMP18} for details. 

\begin{prop}
 Let $A\in\M{d}^+$. Then the following conditions are equivalent: 
 \begin{enumerate}[label=(\roman*)]
    \item $S_A$ is 2-EB.
    \item $S_A$ is $k$-EB, where $2<k<d$.
    \item $S_A$ is EB.
 \end{enumerate}
\end{prop}

\begin{proof}
 Clearly $(iii)\Rightarrow (ii)\Rightarrow (i)$. Now, to prove $(i)\Rightarrow (iii)$, assume that $S_A$ is $2$-EB. To show that $S_A$ is EB, it is enough to prove that $A$ is a diagonal matrix. Given $1\leq i<j\leq d$ let $V_{ij}\in\M{2\times d}$ be the matrix with $(1,i)$ and $(2,j)$ entry equals one and zero elsewhere. By Theorem \ref{thm-nEB-char}, $S_A\circ Ad_{V_{ij}}$ is EB, and hence  
 $$(\id\otimes\T)C_{S_A\circ Ad_{V_{ij}}}=\Matrix{a_{ii}E_{ii}& a_{ij}E_{ji}\\ a_{ji}E_{ij}&a_{jj}E_{jj}}\in (\M{2}\otimes\M{d})^+,$$
  where $E_{kl}\in\M{d}$ are the matrix units. But this is possible only if $\Matrix{0&a_{ij}\\ a_{ji}&0}\in\M{2}^+$, and which implies that $a_{ij}=0=a_{ji}$. Thus $A$ is a diagonal matrix, and conclude that $S_A$ is EB. 
\end{proof}

\begin{prop}
 Given $V\in\M{d_1\times d_2}$ the following conditions are equivalent:
 \begin{enumerate}[label=(\roman*)]
     \item $Ad_V$ is $2$-EB.
     \item $Ad_V$ is $k$-EB, where $2<k<d$.
     \item $Ad_V$ is EB.  
 \end{enumerate} 
\end{prop}

\begin{proof}
 To prove the non-trivial direction $(i)\Rightarrow (iii)$, assume that $Ad_V$ is $2$-EB. We show that rank$(V)=1$ so that $Ad_V$ is EB. If possible assume that $k=$rank$(V)>1$. Let $W\in\M{2\times d_1}$ be such that  rank$(WV)>1$. Then $Ad_{WV}:\M{2}\to\M{d_2}$ is not EB. But,  Theorem \ref{thm-nEB-char}\ref{nEB-CP} implies that $Ad_{V}\circ Ad_W=Ad_{WV}$ is EB, which is a contradiction. Hence rank$(V)=1$. This completes the proof. 
\end{proof}

 In the following, we consider linear maps satisfying some  equivariance property, which is significant in the study of $k$-positivity of linear maps defined on matrix spaces (\cite{COS18}) and entanglement detection (\cite{BCS20}). A linear map $\Phi:\M{d_1}\to\M{d_2}$ is said to be \emph{equivariant}, if for every unitary $U\in\M{d_1}$  there exists a  (not necessarily unitary) matrix  $V(U)\in\M{d_2}$ such that $\Phi\circ Ad_U=Ad_{V(U)}\circ\Phi$. 
  
\begin{prop}[\cite{COS18}]\label{prop-k-pos-char} 
 Let $\Phi:\M{d_1}\to\M{d_2}$ be an equivariant  map and $1\leq k\leq d_1$. Then the following conditions are equivalent: 
 \begin{enumerate}[label=(\roman*)]
    \item $\Phi$ is $k$-positive. 
    \item The principle $(k\times k)$-block submatrix $C_{\Phi(1,2,\cdots,k)}$ is positive.
    \item $\Phi\circ Ad_P$ is CP, where $P=\Matrix{I_{k}&\text{\large 0}_{k\times d_1-k}}\in\M{k\times d_1}$. 
 \end{enumerate}
\end{prop}
  
  The characterization (iii) is not there in \cite{COS18}. But through the same lines of proof of the following theorem, which is an analogue of the above result for $k$-EB maps, we can get a simple alternative proof of Proposition \ref{prop-k-pos-char}. We leave the details to the reader.  
  
\begin{thm}\label{thm-nEB-char-3} 
 Let $\Phi:\M{d_1}\to\M{d_2}$ be an equivariant $k$-positive map, where $1\leq k\leq d_1$. Then the following conditions are equivalent: 
 \begin{enumerate}[label=(\roman*)]
    \item $\Phi$ is $k$-EB. 
    \item The principal $(k\times k)$-block submatrix $C_{\Phi(1,2,\cdots,k)}\in(\M{k}\otimes\M{d_2})^+$ is separable.
    \item $\Phi\circ Ad_P$ is entanglement breaking, where $P=\Matrix{I_{k}&\text{\large 0}_{k\times d_1-k}}\in\M{k\times d_1}$. 
 \end{enumerate}
\end{thm}

\begin{proof}
 $(i)\Rightarrow(ii)$  Follows from Theorem \ref{thm-kEB-Choi-block}.  \\
 $(ii)\Rightarrow(iii)$  Let $E_{ij}^{(d_1)}\in\M{d_1}$ and $E_{ij}^{(k)}\in\M{k}$ be the standard  matrix units. Consider the Choi matrix 
 \begin{align*}
    C_{\Phi\circ Ad_P}
        =\sum_{i,j=1}^kE_{ij}^{(k)}\otimes\Phi(P^*E_{ij}^{(k)}P)
        =\sum_{i,j=1}^kE_{ij}^{(k)}\otimes\Phi(E_{ij}^{(d_1)})
        =C_{\Phi(1,2,\cdots,k)},
 \end{align*}
 which is separable by assumption, and hence $\Phi\circ Ad_P$ is EB. \\  
 $(iii)\Rightarrow (i)$ Let $Q\in\M{k\times d_1}$. By singular value decomposition, there exist unitary matrices $W\in\M{k}, U\in\M{d_1}$ and a rectangular matrix $\Sigma=\Matrix{D&\large{\text 0}_{k\times(d_1-k)}}\in\M{k\times d_1}$, where $D\in\M{k}$ is a diagonal matrix, such that  $Q=W\Sigma U$. Set $\wtilde{W}=WD\in\M{k}$.  Then $\wtilde{W}PU=W\Sigma U=Q$. Now since $\Phi$ is equivariant, corresponding to the unitary $U\in\M{d_1}$ there exists $V(U)\in\M{d_2}$ such that $\Phi\circ Ad_U=Ad_{V(U)}\circ\Phi$.    Hence
 \begin{align*}
    \Phi\circ Ad_Q
        =\Phi\circ Ad_U\circ Ad_P\circ Ad_{\wtilde{W}} 
        =Ad_{V(U)}\circ\Phi\circ Ad_P\circ Ad_{\wtilde{W}}.
 \end{align*}
 By assumption $\Phi\circ Ad_P$ is EB, and hence  from the above equation, it follows that $\Phi\circ Ad_Q$ is also EB. Since $Q\in\M{k\times d_1}$ is arbitrary, by Theorem \ref{thm-nEB-char},  $\Phi$ is $k$-EB.     
\end{proof}

\begin{rmk}
 In the above theorem, $(i)\Rightarrow(ii)\Rightarrow(iii)$ holds for every $\Phi\in\mcl{P}_k(d_1,d_2)$. 
\end{rmk}

  Next, we consider a particular example of an equivariant map, namely the Holevo-Werner map (\cite{WeHo02}). Given $\lambda\in\mbb{R}$ the Holevo-Werner map $\mcl{W}_\lambda:\M{d}\to\M{d}$ is given by 
 $$\mcl{W}_\lambda(X)=\tr(X)I-\lambda X^{\T}.$$
  It is well-known (\cite{CMW19,WeHo02}) that  
  \begin{gather*}
     \mcl{W}_\lambda \mbox{ is CP }\Longleftrightarrow \lambda\in [-1,1]\\
     \mcl{W}_\lambda \mbox{ is PPT }\Longleftrightarrow \mcl{W}_\lambda \mbox{ is EB }\Longleftrightarrow \lambda\in [-1,1/d]\\
     \mcl{W}_\lambda \mbox{ is 2-EB }\Longleftrightarrow \lambda\in [-1,1/2].
 \end{gather*}
  In the following, we discuss when do they become $k$-EB for $k>2$. First, we prove the following technical lemma.   

\begin{lem}\label{prop-direct sum of eb}
 Let $\Phi_i:\M{d}\to\M{d_i},~i=1,2$  and $\Phi:\M{d}\to\M{d_1+d_2}$ be linear maps such that 
 $$\Phi(X)=\Matrix{\Phi_1(X) & 0\\
                   0 & \Phi_2(X)}$$ for all $X\in\M{d}$.
Then $\Phi$ is $k$-EB if and only if $\Phi_i$'s are $k$-EB.
\end{lem}

\begin{proof}
 Assume that $\Phi$ is $k$-EB. Let $\Psi\in\mcl{CP}(k,d)$. Then $\Phi\circ\Psi$ is EB, and hence $Ad_{V_i^*}\circ\Phi\circ\Psi$ is also EB, where
 \begin{align*}
   V_1=\Matrix{I_{d_1}&\text{\large{0}}_{d_1\times d_2}}\in\M{d_1\times(d_1+d_2)}
   \quad\mbox{and}\quad
   V_2=\Matrix{\text{\large{0}}_{d_2\times d_1}&I_{d_2}}\in\M{d_2\times(d_1+d_2)}. 
 \end{align*}
 But $Ad_{V_i^*}\circ\Phi\circ\Psi=\Phi_i\circ\Psi,~i=1,2$. Since $\Psi$ is arbitrary we conclude that $\Phi_i$'s are $k$-EB.  Conversely, assume that both $\Phi_1$ and $\Phi_2$ are $k$-EB, and let $\Psi\in\mcl{CP}(k,d)$. Then $\Phi_i\circ\Psi$'s are EB-maps, and hence there exist $F_j^{(i)}\in\M{k}^+$ and $R_j^{(i)}\in\M{d_i}^+$ such that 
 \begin{align*}
   \Phi_i\circ\Psi(X)=\sum_{j=1}^{r_i}\tr\big(XF_j^{(i)}\big)R_j^{(i)}
 \end{align*} 
 for all $X\in\M{k}$ and $i=1,2$. Then  
 \begin{align*}
   \Phi\circ\Psi(X)=\sum_{j=1}^{r_1}\tr\big(XF_j^{(1)}\big)\Matrix{R_j^{(1)}&\\&\text{\large{0}}_{d_2}}+\sum_{j=1}^{r_2}\tr\big(XF_j^{(2)}\big)\Matrix{\text{\large{0}}_{d_1}&\\&R_j^{(2)}},
 \end{align*}
 for all $X\in\M{k}$ so that  $\Phi\circ\Psi$ is EB. Since $\Psi$ is arbitrary $\Phi$ is $k$-EB. 
\end{proof}

\begin{thm}\label{thm-Wer-Hol-nEB}
 Given $1\leq k\leq d$ the following conditions are equivalent:
  \begin{enumerate}[label=(\roman*)]
     \item $\mcl{W}_\lambda$ is a $k$-entanglement breaking CP-map.
     \item $\mcl{W}_\lambda$ is a $k$-PPT map.
     \item $\lambda\in[-1, 1/k]$. 
  \end{enumerate}. 
\end{thm} 
 
\begin{proof}  
 $(i)\Leftrightarrow (iii)$ We know that $\mcl{W}_\lambda$ is CP if and only if $\lambda\in [-1,1]$.  Also since $\mcl{W}_\lambda$ is equivariant, by Theorem \ref{thm-nEB-char-3}, $\mcl{W}_\lambda$ is $k$-EB if and only if $\mcl{W}_\lambda\circ Ad_P$ is EB, where  $P=\Matrix{I_k & \text{\large{0}}_{k\times (d-k)}}\in \M{k\times d}$. But  
 \begin{align}\label{eq-W}
     \mcl{W}_\lambda\circ Ad_P(X)=\Matrix{\tr(X)I_k-\lambda X^{\T} & 0\\
                              0& \tr(X)I_{d-k}} 
 \end{align}
 for all $X\in\M{k}$. Since the map $\M{k}\ni X\mapsto \tr(X)I_{d-k}\in\M{d-k}$ is always EB, by Lemma \ref{prop-direct sum of eb}, $\mcl{W}_\lambda\circ Ad_P$ is EB if and only if the Holevo-Werner map  $\M{k}\ni X\mapsto \tr(X)I_k-\lambda X^{\T}\in\M{k}$ is EB, and which is true if and only if $\lambda\in [-1,\frac{1}{k}]$. \\ \\
 $(ii)\Leftrightarrow (iii)$  Since $\mcl{W}_\lambda$ is equivariant, by Proposition \ref{prop-k-pos-char}, $\mcl{W}_\lambda$ is $k$-PPT if and only if $\mcl{W}_\lambda\circ Ad_P$ is PPT, where  $P=\Matrix{I_k & \text{\large{0}}_{k\times (d-k)}}\in \M{k\times d}$. But, equation \eqref{eq-W} implies that $\mcl{W}_\lambda\circ Ad_P$ is PPT if and only if the Holevo-Werner map  $\M{k}\ni X\mapsto \tr(X)I_k-\lambda X^{\T}\in\M{k}$ is PPT if and only if $\lambda\in [-1,\frac{1}{k}].$
\end{proof}
 
\begin{note} 
 Equivalence of $(i)$ and $(iii)$ in the above theorem was observed  in  \cite{ChCh20} also.
\end{note}  
 
\begin{eg}[\cite{ChCh20}]\label{eg-EBCP-not-sym}
 The Holevo-Werner map $\mcl{W}_{\frac{1}{k}}:\M{d}\to\M{d}$ is a $k$-entanglement breaking CP-map for $1<k< d$. The composition $\mcl{W}_{\frac{1}{k}}\circ\T=\T\circ \mcl{W}_{\frac{1}{k}}$ is $k$-entanglement breaking but not  $(k+1)$-positive, hence it is not $(k+1)$-entanglement breaking.  
\end{eg}


\begin{cor}
 Let $\Gamma:\M{d}\to \M{d}$ be a positive map and $k\geq 1$. If  $\Gamma\in (\mcl{EB}_k(d))^\circ$, then the following holds:
 \begin{enumerate}[label=(\roman*)]
    \item $-(I_d\otimes \Gamma(I_d))\leq C_\Gamma\leq k (I_d\otimes\Gamma(I_d))$.
    \item $-(I_d\otimes \Gamma(I_d))\leq C_{\Gamma\circ\T}\leq k (I_d\otimes\Gamma(I_d))$.
    \item $-\tr(\Gamma(I_d))\leq\tr(C_{\Gamma}C_{\id_d})\leq k\tr(\Gamma(I_d))$.
    \item $-\tr(\Gamma(I_d))\leq\tr(C_{\Gamma}C_{\T})\leq k\tr(\Gamma(I_d))$.
 \end{enumerate}
\end{cor}

\begin{proof}
 By Theorem \ref{thm-dual-cone},  $\Gamma\circ\Theta^{*}$ is CP for all $\Theta\in\mcl{EB}_k(d)$. Since $\mcl{W}_{\lambda}$ and $\T\circ\mcl{W}_{\lambda}=\mcl{W}_{\lambda}\circ\T$ are self-adjoint $k$-EB maps for all $\lambda\in[-1,\frac{1}{k}]$, we have $\Gamma\circ\mcl{W}_{\lambda}$ and $\Gamma\circ(\mcl{W}_{\lambda}\circ\T)$ are CP-maps. Therefore,
 \begin{align*}
    C_{\Gamma\circ\mcl{W}_{\lambda}}=(I_d\otimes\Gamma(I_d))-\lambda C_{\Gamma\circ\T}\geq 0
    \quad\mbox{and}\quad
    C_{\Gamma\circ\mcl{W}_{\lambda}\circ\T}=(I_d\otimes\Gamma(I_d))-\lambda C_{\Gamma}\geq 0.
 \end{align*}
 Taking $\lambda=-1,\frac{1}{k}$ in the above inequalities we get  $(i)$ and $(ii)$. Now, from $(i)$, 
 \begin{align*}
     &-\bip{\Omega_d,(I_d\otimes \Gamma(I_d))\Omega_d}
      \leq \bip{\Omega_d,C_\Gamma\Omega_d}
      \leq k \bip{\Omega_d,(I_d\otimes\Gamma(I_d))\Omega_d}\\
      &\Longrightarrow
      -\tr\big((I_d\otimes\Gamma(I_d))\ranko{\Omega_d}{\Omega_d}\big)
      \leq\tr(C_\Gamma\ranko{\Omega_d}{\Omega_d})
      \leq k\tr\big((I_d\otimes\Gamma(I_d))\ranko{\Omega_d}{\Omega_d}\big)\\
      &\Longrightarrow 
      -\tr(\Gamma(I_d))\leq \tr(C_\Gamma C_{\id})\leq k\tr(\Gamma(I_d)).
 \end{align*} 
  Similarly, $(iv)$ follows from $(ii)$.
\end{proof}  

\begin{thm}\label{thm-Werner-modified}
 Let $1< k\leq d$ and $\Gamma:\M{d}\to\M{d}$ be a positive map. If $\lambda\in[-\frac{1}{k\norm{\Gamma}},\frac{1}{\norm{\Gamma}}]$, then $\mcl{W}_{\lambda,\Gamma}:\M{d}\to\M{d}$ given by 
 \begin{align}\label{eq-Werner-modified}
    \mcl{W}_{\lambda,\Gamma}(X):=\tr(X)I_d+\lambda \Gamma(X)
 \end{align}
 is a $k$-entanglement breaking map. 
\end{thm}

\begin{proof}
 Let $\lambda\in[-\frac{1}{k\norm{\Gamma}},\frac{1}{\norm{\Gamma}}]$. Since $\alpha:=-\lambda\norm{\Gamma}\in [-1,\frac{1}{k}]$ Remark \ref{rmk-nEB}(ii) and Theorem \ref{thm-Wer-Hol-nEB} implies that  $\Gamma\circ\T\circ W_\alpha$ is a $k$-EB map. Let $\Phi\in\mcl{EB}(d)$ be the map given by $\Phi(X):=\tr(X)R$, where $R=\norm{\Gamma}I_d-\Gamma(I_d)\in\M{d}^+$. Now it follows that 
 \begin{align*}
     W_{\lambda,\Gamma}=\frac{1}{\norm{\Gamma}}\big(\Phi+\Gamma\circ\T\circ W_\alpha\big)
 \end{align*}
 is a $k$-EB map.
\end{proof}

 The following is a generalization of \cite[Theorem 3.3]{CMW19}.
 
\begin{thm}\label{thm-nEB-suff-condn}
 Let $1<k\leq d$ and $\Gamma:\M{d}\to\M{d}$ be a trace preserving positive map such that 
   $$\norm{\tr(X)I_d-\Gamma(X)}_\infty\leq \frac{1}{k}\norm{X}_\infty$$
 for all $X\in\M{d}$. Then $\Gamma$ is a $k$-entanglement breaking map.  
\end{thm}

\begin{proof}
 Let $\wtilde{\Gamma}:=W_{-1,\Gamma}:\M{d}\to\M{d}$, which is a positive map. Note that $\bnorm{\wtilde{\Gamma}}\leq \frac{1}{k}$, hence by Theorem \ref{thm-Werner-modified}, $W_{-1,\wtilde{\Gamma}}$ is $k$-EB. But $W_{-1,\wtilde{\Gamma}}=\Gamma$. 
\end{proof}

 In the rest of this section, we consider a particular case of the maps defined by \eqref{eq-Werner-modified}. Given $\lambda\in\mbb{R}$ define $\Phi_{\lambda,d}:\M{d}\to\M{d}$ by   
  \begin{align}\label{eq-Phi-lambda}
      \Phi_{\lambda,d}(X):=\tr(X)I+ \lambda(X+X^{\T})
  \end{align}
 for all $X\in\M{d}$. Clearly $\Phi_{\lambda,d}=\Phi_{\lambda,d}\circ\T$, hence $\Phi_{\lambda,d}$ is CP if and only if $\Phi_{\lambda,d}$ is PPT. Note that $\Phi_{\lambda,d}$ is unital if and only if $d+2\lambda=1$ if and only if $\Phi_{\lambda,d}$ is trace preserving. 
 
\begin{thm}\label{thm-Phi-lambda}
 Let $\lambda\in\mbb{R}, 1<k\leq d$. Then the following statements holds. 
 \begin{enumerate}[label=(\roman*)]
     \item\label{Phi-lambda-pos} 
             $\Phi_{\lambda,d}$ is positive if and only if $\lambda\in [-\frac{1}{2},\infty)$.
     \item \label{Phi-lambda-CP}
             $\Phi_{\lambda,d}$ is EB if and only if $\Phi_{\lambda,d}$ is CP if and only if $\lambda\in[-\frac{1}{d+1},1]$. 
     \item \label{Phi-lambda-nEB}
             $\Phi_{\lambda,d}$ is $k$-EB implies $\Phi_{\lambda,k}$ is EB (and hence $\lambda\in[-\frac{1}{k+1},1]$). Conversely if $\lambda\in[\frac{-1}{2k},1]$, then $\Phi_{\lambda,d}$ is $k$-EB.
 \end{enumerate}
\end{thm}

\begin{proof}
  $\ref{Phi-lambda-pos}$ Assume $\Phi_{\lambda,d}$ is positive. Then $\Phi_{\lambda,d}(E_{11})\geq 0$, which implies that $\lambda\geq -\frac{1}{2}$. Conversely assume that $\lambda\geq -\frac{1}{2}$. Then given a unit vector $u\in\mbb{C}^d$ we have
  \begin{align*}
      \bip{x,\Phi_{\lambda,d}(\ranko{u}{u})x} 
         = \ip{x,x}+\lambda\big(\abs{\ip{x,u}}^2+\abs{\ip{x,\ol{u}}}^2\big)
         \geq \norm{x}^2-\frac{1}{2}\big(\abs{\ip{x,u}}^2+\abs{\ip{x,\ol{u}}}^2\big)
         \geq 0
  \end{align*}
 for all $x\in\mbb{C}^d$. Thus $\Phi_{\lambda,d}(\ranko{u}{u})\geq 0$. From spectral theorem it follows that $\Phi_{\lambda,d}$ is positive. \\ \\
 $\ref{Phi-lambda-CP}$ Clearly, $\Phi_{\lambda,d}$ is EB implies it is CP. Now assume that $\Phi$ is  a CP-map.   Note that 
  \begin{align*}
       C_{\Phi_{\lambda,d}} =(I_d\otimes I_d)+\lambda(\ranko{\Omega_d}{\Omega_d}+(\id\otimes\T)\ranko{\Omega_d}{\Omega_d}).
  \end{align*}
   Then $0\leq \ip{\Omega_d, C_{\Phi_{\lambda,d}}\Omega_d}=d(1+(d+1)\lambda)$, hence $1+(d+1)\lambda\geq 0$, i.e., $\lambda\geq -\frac{1}{d+1}$. Also, since the principal submatrix $\sMatrix{I+2\lambda E_{11} & \lambda(E_{12}+E_{21})\\ \lambda(E_{21}+E_{12})& I+2\lambda E_{22}}$ of $C_{\Phi_{\lambda,d}}$ is positive
 we have $\lambda^2\leq 1$, hence $\lambda\leq 1$.  Conversely, assume that  $\lambda\in [-\frac{1}{d+1},1]$. Then, by Corollary \ref{cor-appndx}, the Choi matrix $C_{\Phi_{\lambda,d}}$ is separable.  Therefore, $\Phi_{\lambda,d}$ is EB. \\ \\
  $\ref{Phi-lambda-nEB}$ Assume that $\Phi_{\lambda,d}$ is $k$-EB. Let $P=\Matrix{I_k&0_{k\times(d-k)}}\in\M{k\times d}$. Then, by Theorem \ref{thm-nEB-char}, $\Phi_{\lambda,d}\circ Ad_P$ is EB, and hence $Ad_{P^*}\circ\Phi_{\lambda,d}\circ Ad_P=\Phi_{\lambda,k}$ is also EB. So, from \ref{Phi-lambda-CP}, it follows that $\lambda\in[\frac{-1}{k+1},1]$.   Conversely, assume that $\lambda\in[\frac{-1}{2k},1]$. If $\lambda\in[0,1]$, then from \ref{Phi-lambda-CP}, we have $\Phi_{\lambda,d}$ is EB and hence $k$-EB. Now if $\lambda\in[\frac{-1}{2k},0]$, then consider the positive map $\Gamma(X)=X+X^{\T}$ on $\M{d}$. As $\norm{\Gamma}=2$, by Theorem \ref{thm-Werner-modified},  $\Phi_{\lambda,d}=\mcl{W}_{\lambda,\Gamma}$ is $k$-EB. 
\end{proof}

 In the above theorem, we believe that $\Phi_{\lambda,d}$ is $k$-EB if and only if $\Phi_{\lambda,k}$ is EB  if and only if $\lambda\in[-\frac{1}{k+1},1]$. But, we could not prove our claim. 

\begin{eg}\label{eq-nEB-nonCP}
 Given $k<\frac{d+1}{2}$, choose $\lambda\in [\frac{-1}{2k}, \frac{-1}{d+1})$. Then the map $\Phi_{\lambda,d}:\M{d}\to\M{d}$ is $k$-EB but not CP.
\end{eg}

 Let $\tr_1,\tr_2$ be the \emph{partial trace maps} on $\M{d_1}\otimes\M{d_2}$, i.e.,  
      $\tr_1(A\otimes B): = \tr(A)B$
      and 
      $\tr_2(A\otimes B):=\tr(B)A$
  for all $A\in\M{d_1}$ and $B\in\M{d_2}$. It is well-known that if $X\in(\M{d_1}\otimes\M{d_2})^+$ is separable then $\tr_i(X),(\id_{d_1}\otimes\T)X,(\T\otimes\id_{d_2})X\geq 0$. If $rank(X)=\max\{d_1,d_2\}$, then $X\in(\M{d_1}\otimes\M{d_2})^+$ is separable iff $(\id\otimes\T)X\geq 0$.  If $X\in(\M{d_1}\otimes\M{d_2})^+$ and $d_1d_2\leq 6$, then $X$ is separable if and only if $(\T\otimes\id)X\geq 0$. See \cite{Per96, HHH96, HLV00} for details. We shall next obtain a necessary condition for separability. Note that, from the above known results, it holds trivially for $\lambda\geq 0$.  
 
\begin{thm}\label{thm-sep-necc}
 Let $X\in(\M{d}\otimes\M{d})^+$ be separable. Then
 \begin{align*}
    &\lambda X+\lambda(\T\otimes \id_d)X+ (I_d\otimes \tr_1(X))\geq 0\\
    &\lambda X+\lambda(\id_d\otimes\T)X+ (\tr_2(X)\otimes I_d)\geq 0
 \end{align*}
 for all $\lambda\in[-\frac{1}{2},\infty)$. 
\end{thm}

\begin{proof}
 Suppose $X=\sum A_i\otimes B_i$ for some $A_i,B_i\in\M{d}^+$. Since $\Gamma(X)=\tr(X)I+\lambda(X+X^{\T})$ is a positive map for all $\lambda\in[-\frac{1}{2},\infty)$, we get
 \begin{align*}
    0&\leq(\Gamma\otimes\id_d)X\\
      &=\sum\big(\tr(A_i)I_d+\lambda(A_i+A_i^{\T})\big)\otimes B_i\\
      &=(I_d\otimes\sum\tr(A_i)B_i)+\lambda(\sum A_i\otimes B_i)+\lambda\sum(A_i^{\T}\otimes B_i)\\
      &=(I_d\otimes\tr_1(X))+\lambda X+\lambda(\T\otimes\id_d)X.
 \end{align*}
 Similarly by considering $(\id_d\otimes\Gamma)(X)$ we can get the second inequality.  
\end{proof}

\begin{cor}\label{cor-EB-necc}
 Let $\Phi:\M{d}\to\M{d}$ be an entanglement breaking CP-map. Then 
  \begin{align*}
      &(I_d\otimes\Phi(I_d))+\lambda(C_\Phi+C_{\Phi\circ\T})\geq 0\\
      &(\tr_2(C_\Phi)\otimes I_d)+\lambda(C_\Phi+C_{\T\circ\Phi})\geq 0
  \end{align*}
 for all $\lambda\in[-\frac{1}{2},\infty)$.
\end{cor}

\begin{proof}
 Follows from the above Theorem as $C_\Phi$ is separable. 
\end{proof}

\section{Schmidt number reducing CP-maps}\label{sec-SN-red-CP}

 Suppose $\Phi\in\mcl{CP}(d_1,d_2)$. Then from the definition of Schmidt number, it follows that 
 $$SN(\Phi\circ\Psi)\leq SN(\Psi)$$
 for all $\Psi\in\mcl{CP}(m,d_1)$, or equivalently, 
 $$SN\big(\id_{m}\otimes\Phi)(X)\big)\leq SN(X)$$
 for all $X\in(\M{m}\otimes\M{d_1})^+$, where $m\geq 1$. We want to know as to when strict inequality occurs in the above inequalities? Note that if $\Psi$ is an EB-map or $SN(X)=1$, then the above inequalities become equality.   

\begin{lem}[{\cite[Lemma 2.1]{CMW19}}]\label{lem-CMW19}
 Given a  $k$-positive map $\Phi:\M{d_1}\to\M{d_2}$ the following conditions are equivalent:
   \begin{enumerate}[label=(\roman*)]
       \item $\Phi$ is $n$-entanglement breaking for $n\leq \min\{k,d_2\}$.
       \item $SN((\id_m\otimes\Phi)(X))\leq \max\{m-n+1, 1\}$, for all $X\in(\M{m}\otimes\M{d_1})^+$ and for all $m\leq \min\{k,d_2\}$.
   \end{enumerate}
\end{lem}

\begin{thm}\label{thm-2EB-char}
 Let $d_1,d_2>1$ and $\Phi:\M{d_1}\to\M{d_2}$ be a CP-map. Then the following conditions are equivalent:
 \begin{enumerate}[label=(\roman*)]
   \item $\Phi $ is $2$-entanglement breaking.
   \item $SN(\Phi \circ \Psi)< SN(\Psi)$ for all non-entanglement breaking CP-maps $\Psi:\M{m}\to\M{d_1}$ and $2\leq m\leq d_2$. 
   \item $SN(\Phi \circ \Psi)< SN(\Psi)$ for all non-entanglement breaking $d_2$-PEB maps $\Psi:\M{m}\to\M{d_1}$ and $m\geq 2$. 
   \item $SN(\Phi \circ \Psi)< SN(\Psi)$ for all non-entanglement breaking CP-maps $\Psi:\M{2}\to\M{d_1}$.   
 \end{enumerate}
\end{thm}

\begin{proof}
  $(i)\Rightarrow (ii)$ Let  $2\leq m \leq d_2$ and $\Psi\in\mcl{CP}(m,d_1)$ be a non-entanglement breaking map. To show that $SN(\Phi\circ\Psi)< SN(\Psi)$. We prove by induction on $m$. When $m=2$, the required inequality follows from Theorem \ref{thm-nEB-char}(ii). Now assume that the result is true for $m-1$, i.e., $ SN(\Phi\circ\Psi)< SN(\Psi)$ for all non-entanglement breaking $\Psi\in\mcl{CP}(l,d_1)$, where $ 2\leq l\leq m-1\leq d_2$. Let $\Psi\in\mcl{CP}(m,d_1)$ be non-entanglement breaking and $r=SN(\Psi)$. If $r=m$, then by Lemma \ref{lem-CMW19} we have 
 \begin{align*}
     SN(\Phi\circ\Psi)=SN((\id_m\otimes\Phi)C_\Psi)\leq m-1 < SN(\Psi).
 \end{align*}
 Now assume that $1<r< m$. Then, by Lemma \ref{lem-SN-decomp}, for all $1\leq j\leq r$ there exist vectors $\psi_{j_i}\in \mbb{C}^j\otimes \mbb{C}^{d_1}$  such that 
 \begin{align*}
   SN(\Phi\circ\Psi) =SN((\id_m\otimes\Phi)(C_{\Psi})) 
                             \leq \max_{ij}\{SN((\id_j\otimes\Phi)(\ranko{\psi_{j_i}}{\psi_{j_i}}))\}
                             =\max_{ij}\{SN(\Phi\circ\Psi_{j_i})\},
\end{align*}
 where $\Psi_{j_i}\in\mcl{CP}(j,d_1)$ is such that $C_{\Psi_{j_i}}=\ranko{\psi_{j_i}}{\psi_{j_i}}$. If $\Psi_{j_i}$ is EB, then $SN(\Phi\circ\Psi_{j_i})=1<r=SN(\Psi)$. If $\Psi_{j_i}$ is not an EB-map, from induction hypothesis, $SN(\Phi\circ\Psi_{j_i})< SN(\Psi_{j_i})\leq j\leq SN(\Psi)$. So, from the above equation, we get $SN(\Phi\circ\Psi)<SN(\Psi)$.\\   
  $(ii)\Rightarrow(iii)$ Suppose $m\geq 2$ and $\Psi\in\mcl{CP}(m,d_1)$ is non-entanglement breaking.  If $m\leq d_2$, then there is nothing to prove. So assume that $m>d_2$ and let $r:=SN(\Psi)\leq d_2$. By Lemma \ref{lem-SN-decomp}, for all $1\leq j\leq r$ there exist $\Psi_{j_i}\in\mcl{CP}(j,d_1)$ and isometries $V_{j_i}:\mbb{C}^j\to\mbb{C}^{m}$ such that 
  \begin{align*}
     C_{\Psi}=\sum_{i,j}(V_{j_i}\otimes I_{d_1})C_{\Psi_{j_i}}(V_{j_i}\otimes I_{d_1})^*.
  \end{align*}
  Observe that not all $\Psi_{j_i}'$s are EB as $\Psi$ is not an EB-map. Since $ j\leq d_2$, by assumption 
  \begin{align*}
      SN(\Phi\circ\Psi)\leq \max_{i,j}SN(\Phi\circ\Psi_{j_i})<\max_{i,j}SN(\Psi_{j_i})\leq j\leq SN(\Psi).
  \end{align*}  
 $(iii)\Rightarrow(iv)$ This follows as $SN(\Psi)=2\leq d_2$ for all non-entanglement breaking $\Psi\in\mcl{CP}(2,d_1)$.\\ 
 $(iv)\Rightarrow(i)$ To prove that $\Phi$ is $2$-EB it is enough to show that $\Phi\circ \Psi$ is EB for every $\Psi\in\mcl{CP}(2,d_1)$. If $\Psi$ is EB, then we are done. Otherwise, by assumption $SN(\Phi\circ\Psi)< SN(\Psi)\leq 2$. Thus, $\Phi\circ\Psi$ has Schmidt number one, and hence it is EB.  
\end{proof}

 One would also ask, given a non-entanglement breaking $\Psi\in\mcl{CP}(d_2,d_1)$, does there exist a map  $\Phi\in\mcl{CP}(d_1,d_2)$ such that $SN(\Phi\circ\Psi)=SN(\Psi)$? If yes, then such a map $\Phi$ cannot be 2-EB. 

\begin{rmk}\label{rmk-SN-attan-CP}
 Suppose $\Psi\in\mcl{CP}(d_2,d_1)$ is a non-entanglement breaking map and $d_1\leq d_2$. Let $\Phi=Ad_V$, where $V\in\M{d_1\times d_2}$ is any co-isometry (i.e., $VV^*=\id_{d_1}$). Then
 \begin{align*}
    SN(\Psi)=SN(Ad_{V^*}\circ Ad_V\circ\Psi)\leq SN(\Phi\circ\Psi)\leq SN(\Psi).
 \end{align*}
 Thus, there exists $\Phi\in\mcl{CP}(d_1,d_2)$ such that $SN(\Phi\circ\Psi)=SN(\Psi)$. But if $d_1>d_2$ this is not the case.  For example, consider a $3$-entanglement breaking CP-map $\Psi_0:\M{d}\to\M{2}$ as given by Theorem \ref{thm-kEB-nonsym}, where $d\geq 4$. Since $\Psi_0^*$ is not $3$-EB there exists $\Psi_1\in\mcl{CP}(3,2)$ such that $\Psi=\Psi_0^*\circ\Psi_1\in\mcl{CP}(3,d)$ is not EB. Now, if possible assume that there exists $\Phi\in\mcl{CP}(d,3)$ such that $SN(\Phi\circ\Psi)=SN(\Psi)$. Since $\Psi_0$ is $3$-EB we have $\Psi_0\circ\Phi^*$ and hence $\Psi_1^*\circ\Psi_0\circ\Phi^*$ are EB-maps. Therefore,  $1=SN(\Phi\circ\Psi)=SN(\Psi)$, which is a contradiction. Thus there is no $\Phi\in\mcl{CP}(d,3)$ such that $SN(\Phi\circ\Psi)=SN(\Psi)$.  
\end{rmk}

\begin{note}
 Recall that the PPT-square conjecture states that the square of a PPT-map is EB; equivalently, the composition of two PPT-maps is EB. It is known (\cite{CMW19,CYT19}) that if $\Phi_1,\Phi_2\in\mcl{PPT}(3,3)$, then $\Phi_1\circ\Phi_2\in\mcl{EB}(3)$. We can prove this fact as follows also. If $\Phi_2$ is EB, then there is nothing to prove. So assume $\Phi_2$ is not EB. Since $\Phi_1$ is PPT, by Theorem \ref{prop-PPT-nEB}, $\Phi_1$ is $2$-EB. Hence, by Lemma \ref{lem-CMW19} and Theorem \ref{thm-2EB-char}, $SN(\Phi_1\circ\Phi_2)<SN(\Phi_2)\leq2$. Thus $SN(\Phi_1\circ\Phi_2)=1$.    
\end{note}

\begin{rmk}
 Let $\Phi:\M{d_1}\to\M{d_2}$ be a PPT-map and $m\geq 1$.
 \begin{enumerate}[label=(\roman*)]
   \item If $d_2=2$ or $3$, then $\Phi\circ\Psi\in\mcl{EB}(m,d_2)$ for all $\Psi\in\mcl{PEB}_2(m,d_1)$. 
   \item If $d_1=2$ or $3$, then $\Psi\circ\Phi\in\mcl{EB}(d_1,m)$ for all $\Psi\in\mcl{PEB}_2(d_2,m)$. 
 \end{enumerate}
 To see $(i)$, observe that by Proposition \ref{prop-PPT-nEB} $\Phi$ is $2$-EB. So if $\Psi\in\mcl{PEB}_2(m,d_1)$ is not EB, then $SN(\Phi\circ\Psi)<SN(\Psi)=2$. Thus $\Phi\circ\Psi$ is EB. Similarly by considering $\Phi^*$ we can get $(ii)$.
\end{rmk}

\begin{note}
 Given any $\Gamma\in\mcl{P}(d_1,d_2)$ and $k>1$ the map $\id_k\otimes\Gamma$ is not $2$-EB. For, suppose $A\in(\M{2}\otimes\M{k})^+$ is entangled and $B\in\M{d_1}^+$. Then $(\id_2\otimes\id_k\otimes\Gamma)(A\otimes B)=A\otimes\Gamma(B)$
 is not separable in $\M{2}\otimes(\M{k}\otimes\M{d_1})$. For, if $A\otimes\Gamma(B)$ is separable and $\tr^{(d_1)}:\M{k}\otimes\M{d_1}\to\M{k}$ is the partial trace map, then
 \begin{align*}
    \tr(\Gamma(B))A=(\id_2\otimes\tr^{(d_1)})(A\otimes\Gamma(B))\in\M{2}^+\otimes\M{k}^+,
 \end{align*}
  which is not possible as $A$ is entangled. 
\end{note}

\begin{thm}\label{thm-nEB+mEB}
 Let $2\leq m, n\leq d$ and $\Phi,\Psi:\M{d}\to\M{d}$ be $n$-entanglement breaking CP-map and $m$-entanglement breaking CP-map, respectively. Then $\Phi\circ\Psi$ is $(n+m-1)$-entanglement breaking CP-map.
\end{thm}

\begin{proof}
 \ul{Case (1):} Suppose $d\leq n+m-1$. Let $X\in(\M{d}\otimes\M{d})^+$. Clearly, $Y=(\id_d\otimes\Psi)(X)\in(\M{d}\otimes\M{d})^+$. Now since $\Psi$ is $m$-EB, by Lemma \ref{lem-CMW19}, $SN(Y)\leq\max\{d-m+1,1\}= n$. Further, since $\Phi$ is $n$-EB, by Theorem \ref{thm-nEB-char}\ref{nEB-SN}, 
 \begin{align*}
     SN\big((\id_d\otimes\Phi\circ\Psi)(X)\big)=SN\big((\id_d\otimes\Phi)(Y)\big)=1.
 \end{align*}
 Thus $(\id_d\otimes\Phi\circ\Psi)(X)$ is separable for all $X\in(\M{d}\otimes\M{d})^+$. Hence $\Phi\circ\Psi$ is EB, and  in particular, $(n+m-1)$-entanglement breaking. \\
 \ul{Case (2):} Suppose $n+m-1< d$.   Let $\Gamma\in\mcl{CP}(n+m-1,d)$. We shall  show that  $\Phi\circ\Psi\circ\Gamma$ is EB so that, by Theorem \ref{thm-nEB-char}\ref{nEB-CP}, $\Phi\circ\Psi$ is $(n+m-1)$-entanglement breaking. Since $\Psi$ is $m$-EB, by Lemma \ref{lem-CMW19}, $SN(\Psi\circ\Gamma)=SN((\id_{n+m-1}\otimes\Psi)C_\Gamma)\leq n$. Thus $\Psi\circ\Gamma\in\mcl{PEB}_n(n+m-1,d)$. Since $\Phi$ is $n$-EB, by Theorem \ref{thm-nEB-char}\ref{nEB-nPEB-m},  $\Phi\circ\Psi\circ\Gamma$ is EB. This completes the proof. 
\end{proof}

 The following improvised version of \cite[Theorem 2.1]{CMW19} is an immediate consequence of the above theorem. 

\begin{cor}\label{cor-nEB+2EB-1}
 Let $2\leq n_i\leq d$ and $k\in\mbb{N}$. If $\Phi_i:\M{d}\to\M{d}$ are $n_i$-entanglement breaking CP-maps for $1\leq i\leq k$, then the composition $\Phi_1\circ\Phi_2\circ\cdots\circ\Phi_k$ is $(n-k+1)$-entanglement breaking CP-map, where $n=\sum_i n_i $. 
\end{cor}

\begin{cor}\label{cor-nEB+2EB-2}
 Let $2\leq k\leq d$ and $\Phi:\M{d}\to\M{d}$ be a $k$-entanglement breaking CP-map. Then $\Phi^m$ is entanglement breaking, where $m=\min\{SN(\Phi),\lceil\frac{d-1}{k-1}\rceil\}$.
\end{cor}

\begin{proof}
 Let $k_1=SN(\Phi)$. Since $\Phi$ is $2$-EB, by Theorem \ref{thm-2EB-char},  $SN(\Phi^{k_1})=1$.  Meanwhile, if $k_2$ is such that $kk_2-k_2+1=d$, then Corollary \ref{cor-nEB+2EB-1}  implies that  $SN(\Phi^{k_2})=1$. Now we conclude that $SN(\Phi^m)=1$, where $m=\min\{SN(\Phi),\lceil\frac{d-1}{k-1}\rceil\}$. 
\end{proof}

\section{Majorization}\label{sec-Maj}
 Given $x=(x_1,x_2,\cdots,x_d)^{\T}\in\mbb{R}^d$, we denote by $x^{\downarrow }=(x_1^\downarrow,x_2^\downarrow,\cdots,x_d^\downarrow)^{\T}\in \mathbb {R} ^{d}$ the vector with the same components, but sorted in descending order. Thus, $x_1^\downarrow\geq x_2^\downarrow\geq \cdots\geq x_d^\downarrow$. Given $x,y\in\mbb{R}^d$, we say that $x$ is \emph{weakly majorized} by $y$ (and write $x\prec_w y$), if 
 \begin{align*}
     \sum _{i=1}^{k}x_i^{\downarrow }\leq\sum _{i=1}^{k}y_{i}^{\downarrow }\qquad\mbox{ for all  }1\leq k\leq d.
 \end{align*}
 If $x\prec_w y$ and $\sum_{i=1}^dx_i=\sum_{i=1}^dy_i$, then we say that $x$ is \emph{majorized} by $y$, and write $x\prec y$.  If the dimensions of $x$ and $y$ are different, we define $x\prec y$ and $x\prec_w y$ similarly by appending extra zeros to the smaller vector to equalize their dimensions. If $A,B$ are two positive matrices, then we say $A$ is \emph{(weakly) majorized} by $B$ if $\sigma(A)$ is (weakly) majorized by $\sigma(B)$, where $\sigma(X)\in\mbb{R}^d$ denotes the vector of all eigenvalues of $X\in\M{d}^+$ arranged in the decreasing order. 
 
  A matrix $D=[d_{ij}]\in\M{d}$ is said to be \emph{doubly sub-stochastic} if $d_{ij}\geq 0$ for all $1\leq i,j\leq d$ and  
 \begin{align*}
    \sum_{i=1}^d d_{ij}\leq 1 \quad\forall~1\leq j\leq d
    \qquad\mbox{and}\qquad
    \sum_{j=1}^d d_{ij}\leq 1 \quad\forall~1\leq i\leq d.
 \end{align*}
 In the above, if $\sum_{i=1}^d d_{ij}=\sum_{j=1}^d d_{ij}=1$, then $D$ is called \emph{doubly stochastic}. It is well-known that $x\prec y$ (resp. $x\prec_w y$) if and only if $x=Dy$ for some doubly stochastic (resp. sub-stochastic) matrix $D\in\M{d}$. 

 It is known (\cite{NiKe01,Hir03}) that $X$ is majorized by both $\tr_1(X)$ and $\tr_2(X)$ whenever $X\in (\M{d}\otimes\M{d})^+$ is separable. Equivalently, if $\Phi:\M{d}\to\M{d}$ is an EB-map, then $C_\Phi$ is majorized by both $\tr_1(C_\Phi)$ and $\tr_2(C_\Phi)$. We prove an analogue of this result for $k$-EB maps, and the proof is almost same as that in \cite{Hir03}.  

 Given $X\in(\M{d}\otimes\M{d})^+=\B{\mbb{C}^d\otimes\mbb{C}^d}^+$ decompose $\mbb{C}^d=\ker{\tr_2(X)}^\perp\bigoplus\ker{\tr_2(X)}$. Then with respect to the decomposition 
 \begin{align*}
    \mbb{C}^d\otimes\mbb{C}^d
    =\Big(\ker{\tr_2(X)}^\perp\otimes\mbb{C}^d\Big)\bigoplus\Big(\ker{\tr_2(X)}\otimes\mbb{C}^d\Big)
 \end{align*}
 we can write 
 \begin{align}\label{eq-X-decomp-1}
    X=\sMatrix{X_1&0\\0&0},
 \end{align}
 where $X_1$ acts on $\ker{\tr_2(X)}^\perp\otimes\mbb{C}^d$. See \cite[Lemma 2]{Hir03} for details. 

 The following is a generalization of \cite[Theorem 1]{Hir03}. 

\begin{thm}
 Let $X\in(\M{d}\otimes\M{d})^+$ and $k\geq 1$.
 \begin{enumerate}[label=(\roman*)]
     \item If $(\id_{d}\otimes W_{-\frac{1}{k},\T})(X)\in (\M{d}\otimes\M{d})^+$, then $X$ is weakly majorized by $k\tr_2(X)$.
     \item If $(W_{-\frac{1}{k},\T}\otimes \id_{d})(X)\in (\M{d}\otimes\M{d})^+$, then $X$ is weakly majorized by $k\tr_1(X)$.
 \end{enumerate}
 (Here $W_{-\frac{1}{k},\T}$ is given by  \eqref{eq-Werner-modified}.)
\end{thm}

\begin{proof}
 Because of \eqref{eq-X-decomp-1}, without loss of generality, we can assume $\tr_2(X)$ is invertible and diagonal, say $\tr_2(X)=\sum_{i=1}^d\alpha_iE_{ii}$, where $\alpha_i\in(0,\infty)$. Since 
 $$(\tr_2(X)\otimes I_d)-\frac{X}{k}=(\id_d\otimes W_{-\frac{1}{k},\T})(X) \geq 0,$$
 by \cite[Lemma 1]{Hir03}, there exists a $C\in\M{d}\otimes\M{d}$ with $\norm{C}\leq 1$ such that 
 \begin{align}\label{eq-sqrt-X}
     X^{\frac{1}{2}}&=\sqrt{k}(\tr_2(X)^{\frac{1}{2}}\otimes I_d)C.
 \end{align}
 Suppose $\sigma(X)=(\lambda_{11},\lambda_{12},\cdots,\lambda_{1d},\cdots,\lambda_{d1},\cdots,\lambda_{dd})^{\T}\in\mbb{R}^{d^2}$. Let  $U\in\M{d}\otimes\M{d}$ be a unitary such that 
 \begin{align}\label{eq-UXU}
    U^*X^{\frac{1}{2}}U
    =diag\Big(\sqrt{\lambda_{11}},\sqrt{\lambda_{12}},\cdots,\sqrt{\lambda_{1d}},\cdots,\sqrt{\lambda_{d1}},\cdots,\sqrt{\lambda_{dd}}\Big).
 \end{align}
 Let $R=CU$. Then from \eqref{eq-sqrt-X}, \eqref{eq-UXU} we  get
 \begin{align}
      X&=k (\tr_2(X)^{\frac{1}{2}}\otimes I_d)CC^*(\tr_2(X)^{\frac{1}{2}}\otimes I_d)           \label{eq-X}\\
   diag (\sqrt{\lambda_{11}},\sqrt{\lambda_{12}},\cdots \sqrt{\lambda_{dd}})
     &=\sqrt{k}U^*(\tr_2(X)^{\frac{1}{2}}\otimes I_d)R            \\ 
 diag (\lambda_{11},\lambda_{12},\cdots,\lambda_{dd})
    &=k R^*(\tr_2(X)\otimes I_d)R.    \label{eq-lambda-ij}
\end{align}
 From \eqref{eq-X}, for all $1\leq i,j\leq d$, we have
 \begin{align}\label{eq-X-entries}
  \ip{e_i\otimes e_j,X e_i\otimes e_j} 
        &=k\bip{e_i\otimes e_j,(\tr_2(X)^{\frac{1}{2}}\otimes I_d)CC^*(\tr_2(X)^{\frac{1}{2}}\otimes I_d) (e_i\otimes e_j)} \notag \\
        &=k\bip{(\tr_2(X)^{\frac{1}{2}}\otimes I_d)(e_i\otimes e_j),CC^*(\tr_2(X)^{\frac{1}{2}}\otimes I_d) (e_i\otimes e_j)}  \notag \\
        &=k\bip{\sqrt{\alpha_i}e_i\otimes e_j,CC^*(\sqrt{\alpha_i}e_i\otimes e_j)} \notag \\
        &=k\alpha_i\bip{e_i\otimes e_j, CC^*(e_i\otimes e_j)} \notag \\
        &=k\alpha_i\Bip{e_i\otimes e_j,C\Big(\sum_{p,q}\bip{e_p\otimes e_q,C^*(e_i\otimes e_j)}e_p\otimes e_q\Big)} \notag \\
        &=k\alpha_i\sum_{p,q=1}^d\abs{\bip{e_i\otimes e_j,C(e_p\otimes e_q)}}^2
 \end{align} 
  Similarly, from \eqref{eq-lambda-ij}, we get 
  \begin{align}\label{eq-lambda-ij-1}
       \lambda_{ij}=k\sum_{p,q=1}^d\alpha_p\abs{\bip{e_p\otimes e_q,R(e_i\otimes e_j)}}^2
  \end{align}
  Now, from the definition of $\tr_2(X)$ and \eqref{eq-X-entries},   we have 
  \begin{align*}
     &\alpha_p=\ip{e_p,\tr_2(X)e_p}
                   =\sum_{l=1}^d\bip{e_p\otimes e_l,X(e_p\otimes e_l)}  
                   =k\alpha_p\sum_{l,m,n=1}^d\abs{\bip{e_p\otimes e_l,C(e_m\otimes e_n)}}^2,\\
    i.e.,\quad &\frac{1}{k}=\sum_{l,m,n=1}^d\abs{\bip{e_p\otimes e_l,C(e_m\otimes e_n)}}^2               
  \end{align*}
  for all $1\leq p\leq d$. Let $S=[S_{ij}]\in\M{d}$ where 
  \begin{align*}
     S_{ij}=\sum_{q=1}^d\abs{\bip{e_j\otimes e_q,R(e_1\otimes e_i)}}^2\qquad\forall~1\leq i,j\leq d.
  \end{align*}
  Clearly $S_{ij}\geq 0$. Further,
  \begin{align*}
    \sum_{i=1}^d S_{ij}
    =\sum_{i,q=1}^d\abs{\bip{e_j\otimes e_q,R(e_1\otimes e_i)}}^2
     \leq \sum_{q,m,n=1}^d \abs{\ip{e_j\otimes e_q, R(e_m\otimes e_n)}}^2
     =\frac{1}{k}
      <1   
  \end{align*}
 for all $1\leq j\leq d$. Also since $\norm{R}\leq 1$ 
  \begin{align*}
      \sum_{j=1}^d S_{ij}
      =\sum_{j,q=1}^d \abs{\ip{e_j\otimes e_q, R(e_1\otimes e_i)}}^2
      =\bip{e_1\otimes e_i,R^*R(e_1\otimes e_i)}
      \leq 1
  \end{align*}
  for all $1\leq i\leq d$.  Thus $S$ is a doubly sub-stochastic matrix. Now from \eqref{eq-lambda-ij-1},
  \begin{align*}
     \Matrix{\lambda_{11}\\ \lambda_{12}\\\vdots\\\lambda_{1d}}=S\Matrix{k\alpha_1\\ k\alpha_2\\\vdots\\k\alpha_d},
  \end{align*}
  and hence $(\lambda_{11},\cdots,\lambda_{1d})^{\T}\prec_w(k\alpha_1,k\alpha_2,\cdots,k\alpha_d)^{\T}$ i.e.,
  \begin{align}\label{eq-maj-1}
       \sum_{j=1}^n\lambda_{1j}\leq\sum_{i=1}^nk\alpha_i\qquad\forall~1\leq n\leq d.
  \end{align}
   Note that 
  \begin{align}\label{eq-maj-2}
      \sum_{i,j=1}^d\lambda_{ij}=\tr(X)=\tr(\tr_2(X))=\sum_{i=1}^d\alpha_i\leq \sum_{i=1}^dk\alpha_i.
  \end{align}
  Take $\alpha_i=0$ for all $d<i\leq d^2$ and $(\beta_1,\beta_2,\cdots,\beta_{d^2})^{\T}=(\lambda_{11},\lambda_{12},\cdots,\lambda_{1d},\cdots,\lambda_{d1},\cdots,\lambda_{dd})^{\T}$. Then  for all  for all $d<n\leq d^2$ we have
  \begin{align}\label{eq-maj-3}
      \sum_{i=1}^{n}\beta_i\leq\sum_{i,j=1}^d\lambda_{ij}=\sum_{i=1}^d\alpha_i=\sum_{i=1}^{n}\alpha_i.
  \end{align}
  Now from \eqref{eq-maj-1}, \eqref{eq-maj-2} and \eqref{eq-maj-3} we conclude that $\sigma(X)\prec_w\sigma(k\tr_2(X))$. Similarly we can prove that $\sigma(X)\prec_w\sigma(k\tr_1(X))$.  
\end{proof}

\begin{thm}
Let $\Phi:\M{d}\to\M{d}$ be a $k$-EB map, where $1\leq k\leq d$. Then $C_{\Phi}$ is weakly majorized by both $(d-k+1)\tr_1(C_\Phi)$ and $(d-k+1)\tr_2(C_{\Phi})$. 
\end{thm}

\begin{proof}
 Let $r=d-k+1$. By Theorem \ref{thm-Werner-modified}, $W_{-\frac{1}{r},\T}$ on $\M{d}$ is  $(d-k+1)$-entanglement breaking CP-map, and  hence, by Lemma \ref{lem-CMW19}, $W_{-\frac{1}{r},\T}$ is  $k$-PEB map. From Theorem \ref{thm-nEB-char}\ref{nEB-tr} it follows that $k$-PEB maps are in the dual of $k$- EB maps. Hence, by Theorem \ref{thm-dual-cone}, $\mcl{W}_{-\frac{1}{r},\T}\circ\Phi=(\Phi^*\circ \mcl{W}_{-\frac{1}{r},\T})^*$ is CP. Thus we have $(I_d\otimes \mcl{W}_{-\frac{1}{r},\T})(C_{\Phi})\geq 0$. Hence by the above theorem we get  $\frac{1}{d-k+1} C_{\Phi}$ is weakly majorized by $\tr_2(C_{\Phi})$. Similarly $C_{\Phi}$ is is weakly majorized by $(d-k+1)\tr_2(C_{\Phi})$ also. 
\end{proof}

\section{Discussion}\label{sec-disc}

  It is known  (\cite{YLT16}) that if $d_1d_2\leq 6$, then every $2$-positive map $\Gamma:\M{d_1}\to\M{d_2}$ is decomposable. But when $4\leq\max\{d_1,d_2\}\leq 9$ (\cite[page 18]{CYT17}) or $d_1,d_2\geq 10$ (\cite{HLLM18, BhOs20}), then there exists $2$-positive map which is not decomposable.   \\

\noindent\textbf{Open Problem 1.} Is every $2$-EB map $\Phi:\M{d_1}\to\M{d_2}$ decomposable? 

\begin{thm}
 Suppose the Problem-1 has an affirmative answer. If $\Phi:\M{d}\to\M{d}$ is a  $2$-entanglement breaking PPT-map, then $\Phi^2$ is entanglement breaking.
\end{thm}

\begin{proof} 
  Suppose $\Phi^2$ is not an EB-map. Then there exists a $\Gamma\in\mcl{P}(d)$ such that $\Gamma\circ\Phi^2$ is not CP. Note that $\Gamma\circ\Phi$ is $2$-EB and hence decomposable. So there exist $\Phi_1,\Phi_2\in\mcl{CP}(d)$ such that $\Gamma\circ\Phi=\Phi_1+\Phi_2\circ\T$. Hence
 \begin{align*}
      \Gamma\circ\Phi^2 
                                        =(\Phi_1+\Phi_2\circ\T)\circ\Phi
                                        =\Phi_1\circ\Phi+\Phi_2\circ(\T\circ\Phi)
 \end{align*}
 is CP, which is a contradiction. 
\end{proof}
 
\begin{eg}
  If $\lambda\in [-1,\frac{1}{2}]$, then the Holevo-Werner map $\mcl{W}_\lambda$ (and hence $\T\circ\mcl{W}_\lambda$) is 2-EB. Note that $\mcl{W}_\lambda$ (resp. $\T\circ\mcl{W}_\lambda$) is a CP-map (resp. co-CP), and hence decomposable.
\end{eg}
 
 \begin{eg} 
 Let $\lambda\in [\frac{-1}{4},1]$. Then the map $\Phi_{\lambda,d}:\M{d}\to\M{d}$ given by \eqref{eq-Phi-lambda} is $2$-EB.  If $\lambda\in [\frac{-1}{4},\frac{1}{2}]$, then $\Phi_{\lambda,d}=\frac{1}{2}\big(\mcl{W}_{-2\lambda}+\mcl{W}_{-2\lambda}\circ\T\big)$, hence we conclude that $\Phi_{\lambda,d}$ is  decomposable.  Now if $\lambda\in [\frac{1}{2},1]$, then $\Phi_{\lambda,d}$ is EB, and in particular decomposable.    
 \end{eg}

\section*{Acknowledgments}
 We are grateful to Alexander  M\"{u}ller-Hermes for his valuable comments, providing a proof of a part of  the Theorem \ref{thm-nEB-dual-cone}, and pointing out the reference \cite{ChCh20}. We thank B. V. Rajarama Bhat for his comments. RD is supported by UGC (University of Grant Commission, India) with ref No 21/06/2015(i)/EU-V. NM thanks NBHM (National Board for Higher Mathematics, India) for financial support with ref No 0204/52/2019/R\&D-II/321. KS is partially supported by the IoE-CoE Project (No. SB20210797MAMHRD008573) from MHRD (Ministry of Human Resource Development, India) and partially by the MATRIX grant (File no. MTR/2020/000584) from SERB (Science and Engineering Research Board, India).

\appendix

\section{}
 Let $a,b,c\in\mbb{C}$ and define $\Phi:\M{d}\to\M{d}$ by $\Phi(X)=a\tr(X)I+bX+cX^{\T}$. We want to know for what values of $a,b,c\in\mbb{C}$ the map $\Phi$ is EB; equivalently when the corresponding  Choi matrix 
  $$a(I_d\otimes I_d)+b\ranko{\Omega_d}{\Omega_d}+c\Delta_d\in\M{d}\otimes\M{d}$$
  is separable? Here  $\Delta_d=(\id_d\otimes\T)\ranko{\Omega_d}{\Omega_d}=\sum_{i,j=1}^dE_{ij}\otimes E_{ji}$. We answer this through similar lines as Stormer (\cite{Sto13}) proved the entanglement property of the map $\tr(\cdot)\Gamma(I)+\Gamma(\cdot)$, where $\Gamma\in\mcl{P}(d)$. Through out we let $G$ denotes the compact group 
 \begin{align*}
  G=\{Ad_{U\otimes U}: U\in \M{d}\mbox{ real orthogonal}\}\subseteq\M{d}\otimes\M{d}=\M{d}(\M{d}),
 \end{align*} 
 where $d>1$. Let $\mu$ denotes the normalized Haar measure on $G$ and let
 \begin{align*}
    Fix(G)=\{A\in\M{d}\otimes\M{d}: Ad_{U\otimes U}(A)=A\mbox{ for all real orthogonal matrices }U\in\M{d} \}.
 \end{align*} 
 Clearly $I_d\otimes I_d\in Fix(G)$. Observe that, for every $x=(x_1,\cdots,x_d)^{\T},y=(y_1,\cdots,y_d)^{\T}\in\mbb{C}^d$
 \begin{align*}
   \Delta_d(x\otimes y)=\sum_{ij}x_je_i\otimes y_ie_j=(\sum_i y_ie_i)\otimes (\sum_j x_je_j)=y\otimes x.
 \end{align*}
 Consequently, $Ad_{U\otimes U}(\Delta_d)(x\otimes y)=y\otimes x=\Delta_d(x\otimes y)$ for every  real orthogonal matrices $U\in\M{d}$.  In other words,  $\Delta_d\in Fix(G)$. Note that 
 \begin{align*}
     Ad_{U\otimes U}(\ranko{\Omega_d}{\Omega_d})
          &=Ad_{U\otimes U}\big((\id_d\otimes \T)\Delta_d\big)  \\
          &=(\id_d\otimes \T)\big(Ad_{U\otimes U}(\Delta_d)\big)  \\
          &=(\id_d\otimes \T)(\Delta_d)  \\
          &=\ranko{\Omega_d}{\Omega_d}, 
 \end{align*}
 for every  real orthogonal matrices $U\in\M{d}$. Thus $\ranko{\Omega_d}{\Omega_d}\in Fix(G)$.
 
\begin{lem}\label{lem-App-d=2}
 Let $d=2$. Then $Fix(G)=\{\alpha(I_2\otimes I_2)+\beta\ranko{\Omega_2}{\Omega_2}+\gamma\Delta_2: \alpha,\beta,\gamma\in\mbb{C}\}$. 
\end{lem} 
 
\begin{proof}
 Let $A=[a_{ij}]\in Fix(G)\subseteq\M{2}(\M{2})$. Consider $U=\sMatrix{1&0\\0&-1}\in\M{2}$. Then $Ad_{U\otimes U}(A)=A$ implies that 
 \begin{align*}
    A=\Matrix{a_{11}&0&0&a_{14}\\
                      0&a_{22}&a_{23}&0\\
                      0&a_{32}&a_{33}&0\\
                      a_{41}&0&0&a_{44}
                      }.
 \end{align*}
 Similarly considering $U=\sMatrix{0&1\\1&0}$ we conclude that  
 \begin{align*}
   A&=\Matrix{a_{11}&0&0&a_{14}\\
                    0&a_{22}&a_{23}&0\\
                    0&a_{23}&a_{22}&0\\
                    a_{14}&0&0&a_{11}\\
                    }
      =\Matrix{a_0&&&\\&a_{22}&&\\&&a_{22}&\\&&&a_0}+
        a_{14}\Matrix{1&0&0&1\\0&0&0&0\\0&0&0&0\\1&0&0&1}+
        a_{23}\Matrix{1&0&0&0\\0&0&1&0\\0&1&0&0\\0&0&0&1}\\
      &=(a_0-a_{22})(E_{11}\otimes E_{11}+E_{22}\otimes E_{22})+a_{22}(I_2\otimes I_2)+a_{14}(\ranko{\Omega_2}{\Omega_2})+a_{23}\Delta_2,
 \end{align*}
 where $a_0=a_{11}-(a_{23}+a_{14})$. Since $A,I_2\otimes I_2,\Delta_2,\ranko{\Omega_2}{\Omega_2}$ are in $Fix(G)$, from above  we get $B=(a_0-a_{22})(E_{11}\otimes E_{11}+E_{22}\otimes E_{22})\in Fix(G)$. Take $U=\frac{1}{\sqrt{2}}\sMatrix{1&1\\-1&1}$. Then $Ad_{U\otimes U}(B)=B$ implies that $a_0-a_{22}=0$. Therefore, $A=a_{22}(I_2\otimes I_2)+a_{14}(\ranko{\Omega_2}{\Omega_2})+a_{23}\Delta_2$.
\end{proof} 
 
\begin{lem}\label{lem-App-d>2}
 Let $d>2$. Given $m\neq n$ let $$P_{\{m,n\}}=\sum_{i,j\in\{m,n\}}\ranko{e_i\otimes e_j}{e_i\otimes e_j}$$ be the projection of $\mbb{C}^d\otimes\mbb{C}^d$ onto $\lspan\{e_i\otimes e_j: i,j\in\{m,n\}\}\cong\mbb{C}^2\otimes\mbb{C}^2$. Then for every $A\in Fix(G)$ there exist unique $\alpha_{mn},\beta_{mn},\gamma_{mn}\in\mbb{C}$ such that 
  $$P_{\{m,n\}}AP_{\{m,n\}}^*=\alpha_{mn}(I_2\otimes I_2)+\beta_{mn}\ranko{\Omega_2}{\Omega_2}+\gamma_{mn}\Delta_2.$$
\end{lem} 

\begin{proof}
 Let $U\in\M{2}$ be a real orthogonal matrix. Then $U$ induce a linear map on $\mcl{W}=\lspan\{e_m,e_n\}$. Decompose $\mbb{C}^d=\mcl{W}\oplus\mcl{W}^\perp$ and write $A=\sMatrix{A_{mn,11}&A_{mn,12}\\A_{mn,21}&A_{mn,22}}\in\M{d}(\M{d})$, where $A_{mn,11}\in\M{2}\otimes\M{2}$.  Consider the real orthogonal matrix $\wtilde{U}=\sMatrix{U&0\\0&I_{\mcl{W}^\perp}}\in\M{d}$. Then
 \begin{align*}
   Ad_{\wtilde{U}\otimes\wtilde{U}}(A)=A
      &\Longrightarrow Ad_U(A_{mn,11})=A_{mn,11}.
 \end{align*}
 Since $U\in\M{2}$ is arbitrary, from Lemma \ref{lem-App-d=2}, we get 
 \begin{align*}
    A_{mn,11}=\alpha_{mn}(I_2\otimes I_2)+\beta_{mn}\ranko{\Omega_2}{\Omega_2}+\gamma_{mn}\Delta_2
 \end{align*}
 where $\alpha_{mn},\beta_{mn},\gamma_{mn}\in\mbb{C}$. But, $A_{mn,11}=P_{\{m,n\}}AP_{\{m,n\}}^*$, where
 \begin{align*}
    P_{\{m,n\}}=\sum_{i,j\in\{m,n\}}\ranko{e_i\otimes e_j}{e_i\otimes e_j},
 \end{align*}
  the projection of $\mbb{C}^d\otimes\mbb{C}^d$ onto $\lspan\{e_i\otimes e_j: i,j\in\{m,n\}\}$. Uniqueness of $\alpha_{mn},\beta_{mn},\gamma_{mn}$ follows from the linear independence of $I_2\otimes I_2, \ranko{\Omega_2}{\Omega_2}, \Delta_2$.
\end{proof}

\begin{prop}
 Let $d>2$. Then $Fix(G)=\{\alpha(I_d\otimes I_d)+\beta\ranko{\Omega_d}{\Omega_d}+\gamma\Delta_d: \alpha,\beta,\gamma\in\mbb{C}\}$. 
\end{prop}  

\begin{proof}
 Let $A\in Fix(G)$. Assume that $A=\sum_{m,n,p,q}a_{(m,p),(n,q)}E_{mp\otimes E_{nq}}$. Thus,
 \begin{align*}
     \bip{e_m\otimes e_n, A(e_p\otimes e_q)}&=a_{(m,p),(n,q)}\\
     \bip{e_m\otimes e_n, (\id_d\otimes\T)A(e_p\otimes e_q)}&=\bip{e_m\otimes e_q,A(e_p\otimes e_n)},
 \end{align*}
 for all $1\leq m,n,,p,q\leq d$. Now let $m,n,p,q\in\{1,2,\cdots,d\}$.\\ 
   \ul{\textsf{Case (1):}} $m\neq n$ or $p\neq q$.\\
 \ul{\textsf{Subcase (i):}} $\{m,n\}\neq \{p,q\}$. 
  Suppose $m,n,p,q$ are distinct. Given $r\in\{m,n,p,q\}$ define the real orthogonal matrix 
   \begin{align}\label{eq-U_r}
      U_r=\Big(\sum_{\substack{i=1,\cdots,d\\ i\neq r}}\ranko{e_i}{e_i}\Big)-\ranko{e_r}{e_r}\in\M{d}.
   \end{align}   
 Note that $U(e_r)=-e_r$ and $U(e_i)=e_i$ for all $i\neq r$. Since  $Ad_{U_r\otimes U_r}(A)=A$, we get 
 \begin{align*}
     \bip{e_m\otimes e_n, A(e_p\otimes e_q)}    
       =\bip{U_re_m\otimes U_re_n, A(U_re_p\otimes U_re_q)} 
       =-\bip{e_m\otimes e_n, A(e_p\otimes e_q)}.
 \end{align*} 
  Hence $a_{(m,p),(n,q)}=0$.   Now suppose $m=p$ or $m=q$ or $n=p$ or $n=q$. With out loss of generality assume that $m=p$. Since $\{m,n\}\neq\{p,q\}$ we have $n\neq q$. Fix $r\in\{n,q\}$ and define the real orthogonal matrix $U_r\in\M{d}$ by \eqref{eq-U_r}. Then  also
  \begin{align*}
  \bip{e_m\otimes e_n, A(e_p\otimes e_q)}    
       =\bip{U_re_m\otimes U_re_n, A(U_re_p\otimes U_re_q)} 
       =-\bip{e_m\otimes e_n, A(e_p\otimes e_q)}
 \end{align*} 
 so that $a_{(m,p),(n,q)}=0$. \\   
 \ul{\textsf{Subcase (ii):}} $\{m,n\}=\{p,q\}$. 
 From Lemma \ref{lem-App-d>2}, there exists scalars $\alpha_{mn},\beta_{mn},\gamma_{mn}\in\mbb{C}$ such that  
 \begin{align*}
      P_{\{m,n\}}AP_{\{m,n\}}^*
        =\alpha_{mn}(I_2\otimes I_2)+\beta_{mn}\ranko{\Omega_2}{\Omega_2}+\gamma_{mn}\Delta_2.
    \end{align*}
    Therefore,
 \begin{align*}
     a_{(m,p),(n,q)}
      &= \bip{e_m\otimes e_n, A(e_p\otimes e_q)}  \notag\\
      &= \bip{P_{\{m,n\}}^*e_m\otimes e_n, AP_{\{m,n\}}^*(e_p\otimes e_q)}  \notag\\
      &= \bip{e_m\otimes e_n, P_{\{m,n\}}AP_{\{m,n\}}^*(e_p\otimes e_q)} \notag\\
      &= \bip{e_m\otimes e_n, \big(\alpha_{mn}(I_2\otimes I_2)+\beta_{mn}\ranko{\Omega_2}{\Omega_2}+\gamma_{mn}\Delta_2\big)(e_p\otimes e_q)}\\
      &=\begin{cases}
                                  \alpha_{mn} & \mbox{ if } p=m\neq n=q\\
                                  \gamma_{mn} & \mbox{ if } q=m\neq n=p.
                           \end{cases} 
 \end{align*}
  \ul{\textsf{Case (2):}} $m=n$ and $p=q$. 
  If $m=n=p=q$, then choose any $r\neq m$. Then, by Lemma \ref{lem-App-d>2},  
  \begin{align*}
       a_{(m,p),(n,q)}
         & =\bip{e_m\otimes e_m,A(e_m\otimes e_m)}\\
         &=\bip{e_m\otimes e_m,P_{rm}AP_{rm}^*(e_m\otimes e_m)}\\
         &=\alpha_{rm}\bip{e_m\otimes e_m,e_m\otimes e_m}+\beta_{rm}\bip{e_m\otimes e_m,\ranko{\Omega_2}{\Omega_2}(e_m\otimes e_m)}\\
         &\qquad+\gamma_{rm}\bip{e_m\otimes e_m,\Delta_2(e_m\otimes e_m)}\\
         &=\alpha_{rm}+\beta_{rm}+\gamma_{rm}.
  \end{align*} 
  Thus $a_{(m,m),(m,m)}=\alpha_{rm}+\beta_{rm}+\gamma_{rm}$ for any $r\neq m$. 
  Now suppose $m=n\neq p=q$. Note that $n\neq p$ and $m\neq q$ but $\{m,q\}=\{n,p\}$.  Hence
  \begin{align*}
      a_{(m,p),(n,q)}
        &=\bip{e_m\otimes e_n,A(e_p\otimes e_q)}\\
        &=\bip{e_m\otimes e_q,(\id_d\otimes \T)(A)(e_p\otimes e_n)}\\
        &=\bip{e_m\otimes e_q,P_{mq}(\id\otimes \T)(A)P_{mq}^*(e_p\otimes e_n)}\\
        &=\bip{e_m\otimes e_q,(\id\otimes \T)P_{mq}(A)P_{mq}^*(e_p\otimes e_n)}\\
        &=\alpha_{mq}\bip{e_m\otimes e_q,e_p\otimes e_n}+\beta_{mq}\bip{e_m\otimes e_q,\Delta_2(e_p\otimes e_n)}\\
        &\qquad+\gamma_{mq}\{e_m\otimes e_q,\ranko{\Omega_2}{\Omega_2}(e_p\otimes e_n)\}\\
        &=\beta_{mq}.
  \end{align*}
  Thus
  \begin{align}\label{eq-A-coef-1}
      a_{(m,p),(n,q)}
       =\begin{cases}
             \alpha_{rm}+\beta_{rm}+\gamma_{rm} & \mbox{ if } m=n=p=q \mbox{ (where } r\neq m)\\
             \alpha_{mn} &\mbox{ if } p=m\neq n=q\\
             \beta_{mq} &\mbox{ if } m=n\neq p=q\\
             \gamma_{mn} &  \mbox{ if } q=m\neq n=p\\
              0 & \mbox{ otherwise}.
         \end{cases}
  \end{align} 
   Now given any $m\neq n$ and $p\neq q$ consider a real orthogonal matrix $U\in\M{d}$ such that $U(e_m)=e_p,U(e_n)=e_q$ and $U^2=I_d$. Clearly, then $U(e_p)=e_m$ and $U(e_q)=e_n$. Further, 
   \begin{align*}
       Ad_{U\otimes U}(A)=A
          &\Longrightarrow P_{\{m,n\}}Ad_{U\otimes U}(A)P_{\{m,n\}}^*=P_{\{m,n\}}AP_{\{m,n\}}^*\\
          &\Longrightarrow \alpha_{pq}=\alpha_{mn}, \beta_{pq}=\beta_{mn},\gamma_{pq}=\gamma_{mn}.
   \end{align*}
  Hence, there exists $\alpha,\beta,\gamma\in\mbb{C}$ such that \eqref{eq-A-coef-1} becomes 
   \begin{align}\label{eq-A-coef-2}
      a_{(m,p),(n,q)}
       =\begin{cases}
             \alpha+\beta+\gamma & \mbox{ if } m=n=p=q \mbox{ (where } r\neq m)\\
             \alpha &\mbox{ if } p=m\neq n=q\\
             \beta &\mbox{ if } m=n\neq p=q\\
             \gamma &  \mbox{ if } q=m\neq n=p\\
              0 & \mbox{ otherwise}.
         \end{cases}
  \end{align}  
  Thus $A=\alpha(I_d\otimes I_d)+\beta(\ranko{\Omega_d}{\Omega_d})+\gamma\Delta_d$.
  \mbox{}\vspace{1cm} 
\end{proof}
 
\begin{defn}
 Define $P:\M{d}\otimes\M{d}\to\M{d}\otimes\M{d}$ by 
 \begin{align*}
     P(A)=\int_G Ad_{U\otimes U}(A)d\mu(U).
 \end{align*}
\end{defn} 
 
\begin{obsr}\label{Obsr-appdx}
 We make the following observations:
 \begin{enumerate}[label=(\roman*)]
       \item  $P$ is a unital positive projection with
                  \begin{align}\label{eq-ranP-1}
                   \ran{P}=Fix(G)=\{\alpha(I_d\otimes I_d)+\beta\ranko{\Omega_d}{\Omega_d}+\gamma\Delta_d: \alpha,\beta,\gamma\in\mbb{C}\},
                    \end{align}
                  Further, $\tr(P(A))=\tr(A)$ for all $A\in\M{d}\otimes\M{d}$.
      \item Let $\msc{C}\subseteq \mcl{P}(d)$ be a mapping cone and let $\Phi\in\msc{C}$. Let $\Psi:\M{d}\to\M{d}$ be the linear map such that $C_{\Psi}=P(C_{\Phi})$. Then for any $\Gamma\in\msc{C}^{\circ}$ we have
   \begin{align*}
    \tr(C_\Psi C_\Gamma)
        &=\tr(P(C_{\Phi})C_\Gamma)\\
        &=\tr\Big(\Big(\int_G Ad_{U\otimes U}(C_{\Phi})d\mu(U)\Big)C_\Gamma\Big)\\
        &=\int_G\tr\big(C_{Ad_U\circ\Phi\circ Ad_{U^*}}C_\Gamma\big)d\mu(U)\\
        &\geq 0,
 \end{align*}
 where the last inequality follows because $Ad_U\circ\Phi\circ Ad_{U^*}\in\msc{C}$. Thus  $\Psi\in {\msc{C}^{\circ\circ}}=\msc{C}$. Since $P$ is a positive projection we conclude that 
 $$P(\{C_{\Phi}:\Phi\in\msc{C}\})=\{C_{\Psi}:\Psi\in\msc{C}\}\bigcap \ran{P}.$$    
 In particular, considering $\msc{C}=\mcl{EB}(d)$, we have      
 \begin{align}\label{eq-ranP-sep}
     P\big( (\M{d}\otimes\M{d})^+_{sep}\big)=(\M{d}\otimes\M{d})^+_{sep}\bigcap\ran{P},
 \end{align}
 where $ (\M{d}\otimes\M{d})^+_{sep}$ denotes the set of separable positive matrices.

   \item Let $X,Y\in\M{d}^+$. Since $Ad_{U\otimes U}(\Delta_d)=\Delta_d$ for every real orthogonal $U\in\M{d}$ we have  
 \begin{align*}
     \tr\big(P(X\otimes Y)\Delta_d\big)
       &=\int_G\tr\big(Ad_{U\otimes U}(X\otimes Y)\Delta_d\big)d\mu(U)    \notag\\
       &=\int_G\tr\big((X\otimes Y)Ad_{U^*\otimes U^*}(\Delta_d)\big)d\mu(U)  \notag\\
       &=\int_G\tr\big((X\otimes Y)\Delta_d\big)d\mu(U)   \notag\\
       &=\tr\big((X\otimes Y)\Delta_d\big).
 \end{align*}
 Similarly,   $\tr\big(P(X\otimes Y)\ranko{\Omega_d}{\Omega_d}\big)=\tr\big((X\otimes Y)\ranko{\Omega_d}{\Omega_d}\big)$.  Thus
 \begin{align}
      \tr\big(P(Z)\Delta_d\big)&=\tr(Z\Delta_d) \label{eq-P-Lambda} \\
      \tr\big(P(Z)\ranko{\Omega_d}{\Omega_d}\big)&=\tr(Z\ranko{\Omega_d}{\Omega_d})\label{eq-P-Omega}
 \end{align}
  for all $Z\in (\M{d}\otimes\M{d})^+_{sep}$.

         \item From \eqref{eq-ranP-1} and \eqref{eq-ranP-sep}, we have $a(I_d\otimes I_d)+b\ranko{\Omega_d}{\Omega_d}+c\Delta_d$ is separable if and only if 
            $$a(I_d\otimes I_d)+b\ranko{\Omega_d}{\Omega_d}+c\Delta_d=P(Z)$$
           for some $Z\in (\M{d}\otimes\M{d})^+_{sep}$. Now since $\Delta_d^2=I_d\otimes I_d, \Delta_d\ranko{\Omega_d}{\Omega_d}=\ranko{\Omega_d}{\Omega_d}=\ranko{\Omega_d}{\Omega_d}\Delta_d$ and $\ranko{\Omega_d}{\Omega_d}^2=d\ranko{\Omega_d}{\Omega_d}$, from \eqref{eq-P-Lambda} and \eqref{eq-P-Omega}, we get
 \begin{align}
    &ad+bd+cd^2=\tr(Z\Delta_d)  \label{eq-1}\\
    &ad+bd^2+cd=\tr(Z\ranko{\Omega_d}{\Omega_d}) \label{eq-2}.
 \end{align}
 Further, $\tr(P(Z))=\tr(Z)$ implies that 
 \begin{align}\label{eq-3}
    ad^2+bd+cd=\tr(Z). 
 \end{align}
 Solving \eqref{eq-1}, \eqref{eq-2} and \eqref{eq-3} for $a,b,c$ we get 
 \begin{align}
    a&=\frac{1}{d^3+d^2-2d}\big((d+1)\tr(Z)-\tr(Z\Delta_d) -\tr(Z\ranko{\Omega_d}{\Omega_d}) \big)     \label{eq-a}\\
    b&=\frac{1}{d^3+d^2-2d}\big((d+1)\tr(Z\ranko{\Omega_d}{\Omega_d})-\tr(Z\Delta_d)-\tr(Z)     \label{eq-c}\\
    c&=\frac{1}{d^3+d^2-2d}\big((d+1)\tr(Z\Delta_d) -\tr(Z\ranko{\Omega_d}{\Omega_d})-\tr(Z)     \label{eq-b}.
 \end{align} 
  \end{enumerate}
\end{obsr} 
  
\begin{thm}\label{thm-appdx}
 Let $x,y\in\mbb{C}^d$.  Then 
 \begin{align*}
    P\big(\ranko{x}{x}\otimes\ranko{y}{y}\big)=a(I_d\otimes I_d)+b(\Delta_d+\ranko{\Omega_d}{\Omega_d})
  \end{align*} 
  is separable, where $a,b\in\mbb{R}$ are given by
   \begin{align*}
    a&=\frac{1}{d^3+d^2-2d}\big((d+1)\norm{x}^2\norm{y}^2-\abs{\ip{x,y}}^2-\abs{\ip{x,\ol{y}}}^2 \big)  \\
    b&=\frac{1}{d^3+d^2-2d}\big((d+1)\abs{\ip{x,\ol{y}}}^2-\abs{\ip{x,y}}^2-\norm{x}^2\norm{y}^2\big)   \\
    c&=\frac{1}{d^3+d^2-2d}\big((d+1)\abs{\ip{x,y}}^2-\abs{\ip{x,\ol{y}}}^2-\norm{x}^2\norm{y}^2\big).
 \end{align*}
 (Note that $a\geq 0$.)
\end{thm} 
 
\begin{proof}
 Let $x,y\in\mbb{C}^d$ and set $X=\ranko{x}{x}$ and $Y=\ranko{y}{y}$. Then, from \eqref{eq-ranP-sep},  $P(X\otimes Y)$ is separable and there exist $a,b,c\in\mbb{C}$ such that  
 $$P(X\otimes Y)=a(I_d\otimes I_d)+b\Delta_d+c\ranko{\Omega_d}{\Omega_d}.$$ 
 From \eqref{eq-a}, \eqref{eq-b} and \eqref{eq-c} we get $a,b,c$ as required.
\end{proof}

\begin{cor}\label{cor-appndx}
 If $\lambda\in[\frac{-1}{d+1},1]$, then $(I_d\otimes I_d)+\lambda(\ranko{\Omega_d}{\Omega_d}+\Delta_d)\in\M{d}\otimes\M{d}$ is separable.
\end{cor}

\begin{proof}
 Let $x_1=r_1e_1, y_1=r_1e_2, x_2=y_2=r_2e_1$, where $r_1=\Big(\frac{d^3+d^2-2d}{d+1}\Big)^{\frac{1}{4}}$ and $r_2=\Big(\frac{d^3+d^2-2d}{d-1}\Big)^{\frac{1}{4}}$. From the above theorem,
 \begin{align*}
    P(\ranko{x_1}{x_1}\otimes\ranko{y_1}{y_1})&=(I_d\otimes I_d)-\frac{1}{d+1}(\ranko{\Omega_d}{\Omega_d}+\Delta_d) \mbox{ and }\\
    P(\ranko{x_2}{x_2}\otimes\ranko{y_2}{y_2})&=(I_d\otimes I_d)+(\ranko{\Omega_d}{\Omega_d}+\Delta_d)   
 \end{align*}
 are separable. Now let $\lambda\in[\frac{-1}{d+1},1]$. Then $\lambda=t(\frac{-1}{d+1})+(1-t)$ for some $t\in [0,1]$, and hence
 \begin{align*}
     tP(\ranko{x_1}{x_1}\otimes\ranko{y_1}{y_1})+(1-t)P(\ranko{x_2}{x_2}\otimes\ranko{y_2}{y_2})
     =(I_d\otimes I_d)+\lambda(\ranko{\Omega}{\Omega_d}+\Delta_d)
 \end{align*}
 is separable.
\end{proof}

\bibliographystyle{alpha} 
 \newcommand{\etalchar}[1]{$^{#1}$}

\end{document}